\documentclass{birkjour}

\renewcommand*{\marginpar}[1]{} 
\usepackage[all]{xy}

\usepackage{amsmath, amsthm, amssymb,latexsym, graphics}

\newcommand{\N}{\mathbb{N}}
\newcommand{\Z}{\mathbb{Z}}

\newcommand{\R}{\mathbb{R}}
\newcommand{\C}{\mathbb{C}}
\newcommand{\E}{\mathbb{E}}

\newcommand{\HI}{H^\infty}
\newcommand{\Hor}{\mathcal{H}}
\newcommand{\Ha}{\Hor^\alpha}

\DeclareMathOperator{\supp}{supp}

\DeclareMathOperator{\Prob}{Prob}
\DeclareMathOperator{\diam}{diam}
\DeclareMathOperator{\dive}{div}

\let\Re=\relax \DeclareMathOperator{\Re}{Re}
\let\Im=\relax \DeclareMathOperator{\Im}{Im}

\newtheorem{thmalt}{Theorem}[section]

\theoremstyle{definition}
\newtheorem{rem}[thmalt]{Remark}

\newtheorem{thm}[thmalt]{Theorem}
\newtheorem{cor}[thmalt]{Corollary}
\newtheorem{lem}[thmalt]{Lemma}
\newtheorem{prop}[thmalt]{Proposition}

\numberwithin{equation}{section}

\title[H\"ormander Functional Calculus for Poisson Estimates]
 {H\"ormander Functional Calculus for Poisson Estimates} 

\author[Ch. Kriegler]{Christoph Kriegler}
\address{
Laboratoire de Math\'ematiques (CNRS UMR 6620)\\
Universit\'e Blaise-Pascal (Clermont-Ferrand 2)\\
Campus des C\'ezeaux\\
63177 Aubi\`ere Cedex\\
France
}
\email{christoph.kriegler@math.univ-bpclermont.fr}

\date{\today}

\allowdisplaybreaks

\begin{document}

\begin{abstract}
The aim of the article is to show a H\"ormander spectral multiplier theorem for an operator $A$ whose kernel of the semigroup $\exp(-zA)$ satisfies certain Poisson estimates for complex times $z.$
Here $\exp(-zA)$ acts on $L^p(\Omega),\,1 < p < \infty,$ where $\Omega$ is a space of homogeneous type with the additional conditions that the volume of balls grows polynomially of exponent $d$ and the measure of annuli is controlled by the corresponding euclidean term.
In most of the known H\"ormander type theorems in the literature, Gaussian bounds and self-adjointness for the semigroup are needed,
whereas here the new feature is that the assumptions are the to some extent weaker Poisson bounds, and $\HI$ calculus in place of self-adjointness.
The order of derivation in our H\"ormander multiplier result is typically $\frac{d}{2},$ $d$ being the dimension of the space $\Omega.$
Moreover the functional calculus resulting from our H\"ormander theorem is shown to be $R$-bounded.
Finally, the result is applied to some examples.
\end{abstract}

\maketitle

\section{Introduction}\label{Sec 1 Intro}

Let $f$ be a bounded function on $(0,\infty)$ and $u(f)$ the operator on $L^p(\R^d)$ defined by $[u(f)g]\hat{\phantom{i}}(\xi) = f(|\xi|^2) \hat{g}(\xi).$
H\"ormander's theorem on Fourier multipliers \cite[Theorem 2.5]{Ho} asserts that $u(f) : L^p(\R^d) \to L^p(\R^d)$ is bounded for any $p \in (1,\infty)$ provided that for some integer $N$ strictly larger than $\frac{d}{2}$
\begin{equation}\label{Equ Intro Hor}
\sup_{R > 0} \int_{R/2}^{2R} \left| t^k f^{(k)}(t)\right|^2 \frac{dt}{t} < \infty \quad (k= 0,\ldots,N).
\end{equation}

This theorem has many generalisations to similar contexts, for example to elliptic and sub-elliptic differential operators $A,$
including sublaplacians on Lie groups of polynomial growth, Schr\"odinger operators and elliptic operators on Riemannian manifolds \cite{DuOS}:
Note first that the above $u(f)$ equals $f(-\Delta),$ the functional calculus of the self-adjoint positive operator $-\Delta.$
Now for a self-adjoint operator $A$, a H\"ormander theorem states that the operator $f(A)$ extends boundedly to $L^p(\Omega),\:1<p<\infty$ for any function $f$ satisfying \eqref{Equ Intro Hor} with suitable $N.$ 
In most of the proofs for a H\"ormander theorem in the literature, the assumption of so called Gaussian bounds plays a crucial role.
That means the following.
Suppose that $A$ acts on $L^p(\Omega),\,1<p<\infty,$ where $(\Omega,\mu,\rho)$ is a space of homogeneous type.
Then the semigroup $(\exp(-tA))_{t \geq 0}$ generated by $A$ has an integral kernel $k_t(x,y)$ such that
\begin{equation}\label{Equ Gaussian intro}
|k_t(x,y)| \leq C \mu(B(y,\sqrt{t}))^{-1} \exp \left(-c \frac{\rho(x,y)^2}{t} \right) \quad (t >0,\,x,y\in \Omega).
\end{equation}
This hypothesis includes many elliptic differential operators.
However there are operators such that the integral kernel of the semigroup satisfies only weaker estimates, see e.g. \cite{GiGr,MMMM,OtE}.
Establishing a H\"ormander theorem for these operators is the issue of the present article.
More precisely, let $\Omega$ be a space of homogeneous type with the additional properties that the volume of balls grows polynomially of exponent $d > 0$ and that the measure of annuli is controlled by (a constant times) the corresponding euclidean term, see \eqref{Equ add prop space 1} and \eqref{Equ add prop space 2} for a precise definition.
Let further $A$ act on $L^p(\Omega)$ such that $(\exp(-zA))_{\Re z > 0}$ has an integral kernel $k_z(x,y)$ such that
\begin{equation}\label{Equ Poisson estimate intro}
|k_z(x,y)| \leq C (\cos \arg z)^{-\beta} \frac{1}{\mu(B(x,|z|))} \frac{1}{|1 + \frac{\rho(x,y)^2}{z^2}|^{\frac{d+1}{2}}} \quad (\Re z > 0,\:x,y \in \Omega)
\end{equation}
holds for some $C,\beta \geq 0.$
If $\Omega = \R^d$ and $\beta = 0,$ the right hand side of this estimate is (a constant times) the absolute value of the complex Poisson kernel which obviously decays slower as $\rho(x,y) \to \infty$
than the Gaussian kernel above.
Under a further hypothesis on the homogeneous space $\Omega,$ see \eqref{Equ add prop space 1} and \eqref{Equ add prop space 2} below, and the presence of an $\HI$ calculus of $A$ on $L^2(\Omega),$ we obtain a H\"ormander theorem of the order $N > \frac{d}{2}+\beta$ for operators $A$ satisfying the above estimate.
The proof relies on the behaviour of the semigroup $\exp(-zA)$ generated by $A$ when the complex parameter $z$ approaches the imaginary axis.
Here simple norm estimates are not sufficient but $R$-bounds of the semigroup are needed.
Our method does not need self-adjointness of $A.$
This is new compared to most of the spectral multiplier results in the literature.
In particular, we give a non-self-adjoint example of a Lam\'e operator to which our main result applies.
Note also that Gaussian estimates as in \eqref{Equ Gaussian intro} and self-adjointness in general yield only a Mihlin calculus \cite[Theorem 7.23]{Ouha}, i.e. bounded spectral multipliers $f(A)$ for $f$ satisfying
\[ \max_{0 \leq k \leq N} |t^k f^{(k)}(t)| < \infty,\]
and $N > \frac{d}{2},$ whereas we show that complex Poisson estimates give a H\"ormander functional calculus, i.e.
$f(A)$ is bounded for $f$ satisfying \eqref{Equ Intro Hor} with $N > \frac{d}{2},$ if $\beta = 0$ in \eqref{Equ Poisson estimate intro}.
This yields better estimates for the Bochner-Riesz means $(1 - A/\lambda)^\nu_+$ on $L^p(\Omega),$ see the discussion in the introduction of \cite{DuOS}.
The difficulty in assumption \eqref{Equ Poisson estimate intro} is to show an estimate for complex times $z,$ something that for Gaussian estimates one gets somehow for free out of an estimate like \eqref{Equ Gaussian intro} and self-adjointness.

In Section \ref{Sec 2 Prelims} we will introduce the necessary background and cite a theorem which allows to pass from $R$-bounds on the semigroup to a H\"ormander functional calculus.
In Section \ref{Sec 3 Main Thm} we state and prove the main result, Theorem \ref{Thm Main} and Corollary \ref{Cor Ha calculus} of this article.
In Section \ref{Sec 4 Examples}, an application to a concrete operator is given, for which a H\"ormander theorem was previously unknown,
and two further examples entering our context are discussed.
Finally, in Section \ref{Sec Proof Lem Prelim}, two proofs of technical lemmas are annexed.

\section{Preliminaries}\label{Sec 2 Prelims}

In this section, we provide the necessary background for the Main Section \ref{Sec 3 Main Thm}.
Let $\omega \in (0,\pi).$
A densely defined and closed operator $A$ on $L^p(\Omega),\,1 < p < \infty,$ is called $\omega$-sectorial if $\sigma(A) \subset \overline{\Sigma_\omega}$ where $\Sigma_\omega = \{ z \in \C^* :\: | \arg z| < \omega\},$
and $\|\lambda (\lambda - A)^{-1}\| \leq C_\theta$ for any $\lambda \in \overline{\Sigma_\theta}^c$ and any $\theta \in (\omega,\pi).$
For an $\omega$-sectorial operator $A$ and a function $f \in \HI_0(\Sigma_\theta) = \{ g : \Sigma_\theta \to \C :\: g \text{ analytic and bounded},\:\exists\:C,\epsilon > 0:\:
|g(z)| \leq C \min(|z|^\epsilon,|z|^{-\epsilon})\}$ where $0 < \omega < \theta < \pi,$ one defines the operator $f(A)$ by
\[ f(A)x = \frac{1}{2\pi i} \int_\Gamma f(\lambda) (\lambda - A)^{-1}x d\lambda. \]
Here, $\Gamma$ is the boundary of $\Sigma_{\frac{\omega + \theta}{2}}$ oriented counterclockwise.
This definition coincides with the self-adjoint calculus if applicable.
If there is a constant $C > 0$ such that $\|f(A)\| \leq C \sup_{|\arg z| < \theta} |f(z)|$ for any $f \in \HI_0(\Sigma_\theta),$ then $A$ is said to have a bounded $\HI(\Sigma_\theta)$ calculus,
or just bounded $\HI$ calculus.
Let $\phi_0 \in C^\infty_c(\frac12,2)$ and for $n \in \Z$ put $\phi_n = \phi_0(2^{-n} \cdot).$
We can and do assume that $\sum_{n \in \Z} \phi_n(t) = 1$ for any $t > 0$ \cite[Lemma 6.1.7]{BeL}.
Now define 
\[\Ha = \{ f : [0,\infty) \to \C :\: \|f\|_{\Ha} = |f(0)| + \sup_{n \in \Z} \| (\phi_n f) \circ \exp \|_{W^\alpha_2(\R)} < \infty \},\]
where $W^\alpha_2(\R)$ is the usual Sobolev space.
For $\alpha > \frac12,$ the space $\Ha$ is a Banach algebra endowed with the norm $\|\cdot\|_{\Ha}.$
This class refines condition \eqref{Equ Intro Hor} in the sense that $f \in \Ha \Longrightarrow f$ satisfies \eqref{Equ Intro Hor}
for $\alpha > N$ and the converse holds for $\alpha < N.$
If $\|f(A)\| \leq C \|f\|_{\Ha}$ for any $f \in \bigcap_{\omega > 0} \HI_0(\Sigma_\omega) \cap \Ha,$ then there exists a bounded homomorphism $\Ha \to B(L^p(\Omega)),\: f \mapsto f(A),$
and $A$ is said to have a bounded $\Ha$ calculus.
If $A$ is moreover self-adjoint on $L^2(\Omega)$ then for $f \in \Ha \subseteq L^\infty(\R_+)$ for some $\alpha > \frac12,$
$f(A)$ is defined twice, but one can show that the definition from the $\HI$ calculus plus density in $\Ha$ and the definition from the self-adjoint spectral calculus coincide.
In particular, our notion of $\Ha$ calculus is the same as in most of the definitions in the literature.

Let $(\epsilon_n)_{n \in \N}$ be a sequence of independent random variables such that $\Prob(\epsilon_n = 1) = \Prob(\epsilon_n = -1) = \frac12,$ i.e. a sequence of independent Rademach\-er variables.
Let $X$ be a Banach space.
A subset $\tau \subset B(X)$ is called $R$-bounded if there exists a constant $C > 0$ such that for any choice of finite families $T_1,\ldots,T_n \in \tau$ and $x_1,\ldots,x_n \in X,$ one has
\[ \left( \E \left\| \sum_{k=1}^n \epsilon_k T_k x_k \right\|_X^2 \right)^{\frac12} \leq C \left( \E \left\| \sum_{k=1}^n \epsilon_k x_k \right\|_X^2 \right)^{\frac12}.\]
The least possible constant is denoted by $R(\tau),$ and $R(\tau) = \infty,$ if no such constant is admitted.
Any $R$-bounded set $\tau$ is norm bounded, i.e. $\sup_{T \in \tau} \|T\| \leq R(\tau),$ but the converse is false in general.
If $X = L^p,\: 1 \leq p < \infty,$ then
\begin{equation}\label{Equ R-bounded Lp}
\left( \E \left\| \sum_{k=1}^n \epsilon_k x_k \right\|_X \right)^{\frac12} \cong \left\| \left( \sum_{k=1}^n |x_k|^2 \right)^{\frac12} \right\|_p
\end{equation}
uniformly in $n$ and $x_1,\ldots,x_n.$
A linear mapping $u : Y \to B(X),$ where $Y$ is a further Banach space is called $R$-bounded if $R(u(y):\:\|y\|_Y \leq 1) < \infty.$
The following proposition gives a condition on the semigroup generated by a sectorial operator $A$ so that $A$ has a $\Hor^\beta$ calculus.

\begin{prop}\label{Prop Prelim Ha calculus}
Let $A$ be an $\omega$-sectorial operator for any $\omega > 0$ defined on an $L^p$ space for some $1 < p < \infty,$ and let $A$ have a bounded $\HI$ calculus.
Suppose that for some $\alpha > 0$ the set $\{ \exp(-e^{i\theta} 2^k t A) :\: k \in \Z \}$ is $R$-bounded for any $t > 0$ and $|\theta| < \frac{\pi}{2},$
with $R$-bound $\lesssim \left(\cos(\theta)\right)^{-\alpha}.$
Then for any $\beta > \alpha + \frac12,$ $A$ has a bounded $\Hor^{\beta}$ calculus.
Moreover, this calculus is an $R$-bounded mapping.
\end{prop}

\begin{proof}
This is proved in the case that $A$ has dense range in \cite[Lemma 4.72 and Proposition 4.79]{Kr}, see also \cite{KrW2}.
This proof for which we give a sketch applies also here.
First one deduces from the assumption of $R$-boundedness of the semigroup that
\[ \{ (1 + |t|)^{-\alpha} (1 + 2^k A)^{-\alpha} \exp(i 2^k t A) : \: t \in \R \} \]
is $R$-bounded with $R$-bound independent of $t \in \R.$
Then for $g \in C^\infty_c(0,\infty)$ a representation formula of $g(2^k A) (1 + 2^k A)^{-\alpha} $ is available, namely
\[ g(2^k A) (1 + 2^k A)^{-\alpha} x = \frac{1}{2\pi} \int_\R \hat{g}(t) ( 1 + |t| )^\beta ( 1 + |t| )^{-\beta} (1 + 2^k A)^{-\alpha} \exp(i2^ktA) x dt. \]
If $\beta > \alpha + \frac12,$ then $( 1 + |t| )^{-\beta} \| ( 1 + 2^k A)^{-\alpha} \exp(i2^k tA)\|$ is dominated by a function in $L^2(\R),$
and if $g$ belongs to $W^\beta_2(\R)$ then also $\hat{g}(t) (1 + |t| )^\beta \in L^2(\R).$
In fact, more can be said.
By \cite[Proposition 4.1, Remark 4.2]{HyVe}, the set
\[ \{ g(2^k A) ( 1 + 2^k A)^{-\alpha} :\: g \in C^\infty_c(0,\infty),\: \|g\|_{W^\beta_2(\R)} \leq 1,\: k \in \Z \} \]
is $R$-bounded.
Next one gets rid of the factor $(1 + 2^k A)^{-\alpha}$ above by using a function $\psi(\lambda) = ( 1 + \lambda)^\alpha \phi(\lambda)$ where $\phi \in C^\infty_c(0,\infty)$
and $\phi(\lambda) = 1$ for $\lambda \in [\frac12 ,2].$
The hypotheses of the proposition imply that $\{ \psi(2^k A) :\: k \in \Z \}$ is $R$-bounded. 
Then
\[ \{ g(2^k A):\: g \in C^\infty_c(0,\infty), \: \supp g \subset [\frac12,2], \: \|g\|_{W^\beta_2(\R)} \leq 1,\: k \in \Z \} \]
is $R$-bounded.
The hypotheses imply moreover that there holds the following equivalences of Paley-Littlewood type:
\[ \| f \|_p \cong \left\| \left( \sum_{k \in \Z} |\phi(2^k A) f|^2 \right)^{\frac12} \right\|_p \cong \left\| \left( \sum_{k \in \Z} |\tilde\phi(2^k A)\phi(2^k A) f|^2 \right)^{\frac12} \right\|_p \]
for a function $\phi \in C^\infty_c(0,\infty),\phi$ not vanishing identically zero, $\supp \phi \subset [\frac12,2]$ and $\tilde\phi = \phi(2^{-1} \cdot) + \phi + \phi(2 \cdot).$
Then one can show that $g(A)$ is bounded for $\|g\|_{\Hor^\beta} < \infty:$
\begin{align*}
 \|g(A) f\|_p & \cong \left\| \left( \sum_{k \in \Z} |\tilde\phi(2^k A) g(A) \phi(2^k A) f|^2 \right)^{\frac12} \right\|_p \\
& \cong \left\| \left( \sum_k |\tilde\phi g(2^{-k}\cdot)(2^k A) \phi(2^k A) f|^2 \right)^{\frac12} \right\|_p \\
& \lesssim R(\{\tilde\phi g(2^{-k} \cdot) : k \in \Z\}) \left\| \left( \sum_{k\in\Z} | \phi(2^kA) f |^2 \right)^{\frac12} \right\|_p \\
& \lesssim \|g\|_{\Hor^\beta} \|f\|_p.
\end{align*}
Thus $\{ g(A):\: \|g\|_{\Hor^\beta} \leq 1 \}$ is a bounded subset of $B(L^p).$
In a similar manner to the calculation right above, using the fact that $L^p$ has Pisier's property $(\alpha),$
one shows that this set is moreover $R$-bounded.
\end{proof}

The space $\Omega$ on which the operator $A$ acts will be a space of homogeneous type.
This means that $(\Omega,\rho)$ is a metric space endowed with a nonnegative Borel measure $\mu$ which satisfies the doubling condition:
There exists a constant $C > 0$ such that for all $x \in \Omega$ and $r>0,$ 
\[\mu(B(x,2r)) \leq C \mu(B(x,r)) < \infty,\]
where we set $B(x,r) = \{ y \in \Omega : \: \rho(x,y) < r \}.$
Note that the doubling condition implies the following strong homogeneity property:
There exists $C > 0$ and a dimension $d > 0$ such that for all $\lambda \geq 1,$ for all $x \in \Omega$ and all $r > 0$ we have
$\mu(B(x,\lambda r)) \leq C \lambda^d \mu(B(x,r)).$
We will assume that the space of homogeneous type $(\Omega,\mu,\rho)$ has the following two additional properties
\begin{align}
\mu(B(x,r)) & \cong r^d \text{ if }\diam(\Omega) = \infty, \nonumber \\
\mu(B(x,r)) & \cong \min(r^d,1)\text{ if }\diam(\Omega) < \infty \label{Equ add prop space 1}
\intertext{and}
\mu(B(x,r,R)) & \leq C (R^d - r^d) \quad (x \in \Omega,\: R > r > 0) \text{ if }\diam(\Omega) = \infty, \nonumber \\
\mu(B(x,r,R)) & \leq C(R-r) \min(R^{d-1},1) \quad (x \in \Omega,\: R > r > \frac12 R > 0)\text{ if }\diam(\Omega) < \infty \label{Equ add prop space 2}
\end{align}
where we denote $\diam(\Omega) = \sup\{\rho(x,y):\:x,y\in \Omega\}$ and $B(x,r,R) = B(x,R) \backslash B(x,r).$
Note that if \eqref{Equ add prop space 1} holds, then $\diam(\Omega) < \infty$ if and only if $\mu(\Omega) < \infty.$

\section{The Main Theorem}\label{Sec 3 Main Thm}

We let $(\Omega,\mu,\rho)$ be a space of homogeneous type with the additional properties \eqref{Equ add prop space 1} and \eqref{Equ add prop space 2}.
We further let $T_z = \exp(-zA)$ be a semigroup on $L^2(\Omega)$ with the properties:
The generator $A$ has a bounded $\HI(\Sigma_\omega)$ calculus for some $\omega \in (0,\pi)$ on $L^2(\Omega)$, and $T_z$ has an integral kernel $k_z(x,y)$ for $\Re z > 0$ i.e. $(T_z f)(x) = \int_\Omega k_z(x,y) f(y) d\mu(y)$ for any $f \in L^2(\Omega).$
We assume that
\begin{equation}\label{Equ add prop semigroup 1}
|k_z(x,y)| \leq C (\cos \arg z)^{-\beta} \frac{1}{\mu(B(x,|z|))} \frac{1}{|1 + \frac{\rho(x,y)^2}{z^2}|^{\frac{d+1}{2}}} \quad (z \in \C_+,\,x,y \in \Omega).
\end{equation}

\begin{prop}\label{Prop A has HI calculus}
Let $(\Omega,\mu,\rho)$ be a space of homogeneous type satisfying \eqref{Equ add prop space 1} and \eqref{Equ add prop space 2} and $T_t = \exp(-tA)$ a semigroup which acts on all $L^p(\Omega),\,1 < p < \infty.$ 
Assume that $T_z$ is analytic on $L^2(\Omega)$ on $z \in \C_+ = \{ \lambda \in \C:\: \Re \lambda > 0\}$ and bounded on each subsector $\Sigma_\omega$ for $\omega \in (0,\frac{\pi}{2}),$ and that $T_z$ has an integral kernel which satisfies \eqref{Equ add prop semigroup 1}.
Assume moreover that $A$ has a bounded $\HI(\Sigma_\omega)$ calculus on $L^2(\Omega)$ for some $\omega \in (0,\pi),$ which is the case e.g. when $A$ is self-adjoint.
Then the operator $A$ has an $\HI(\Sigma_\omega)$ calculus on $L^p(\Omega)$ for any $\omega \in (0,\pi)$ and $1 < p < \infty.$
\end{prop}

\begin{proof}
The proposition follows from \cite[Theorem 3.1]{DuRo}.
Indeed, let $\theta \in (0,\frac{\pi}{2}).$
The kernel $k_z(x,y)$ satisfies on $z \in \Sigma_\theta$ the bound
\begin{align*}
|k_z(x,y)| & \lesssim \frac{(\cos\arg z)^{-\beta}}{\mu(B(x,|z|))}\frac{1}{\left| 1 + \left(\frac{\rho(x,y)}{z}\right)^2 \right|^{\frac{d+1}{2}}} \\
& \lesssim_\theta \mu(B(x,|\Re z|))^{-1} \left(1 + \left(\frac{\rho(x,y)}{\Re z}\right)^2 \right)^{-\frac{d+1}{2}}
\end{align*}
since $|z| \cong \Re z$ for $z \in \Sigma_\theta$ and
$|1 + \left(\frac{\rho(x,y)}{\Re z}\right)^2| \leq 1 + \left|\frac{\rho(x,y)}{\Re z}\right|^2 \lesssim 1 + \left|\frac{\rho(x,y)}{z}\right|^2 \lesssim \left| 1 + \left( \frac{\rho(x,y)}{z} \right)^2 \right|.$
Then with $G_t$ given by \cite[(7)]{DuRo} and $g(x) = c (1 + x^2)^{-\frac{d+1}{2}},$ we can deduce from \cite[Theorem 3.1]{DuRo} that $A$ has a bounded $\HI(\Sigma_\omega)$ calculus on $L^p(\Omega)$
for any $p \in (1,\infty)$ and $\omega > \frac{\pi}{2} - \theta.$
\end{proof}

The following is the main theorem of this article.

\begin{thm}\label{Thm Main}
Let $(\Omega,\mu,\rho)$ be a space of homogeneous type satisfying \eqref{Equ add prop space 1} and \eqref{Equ add prop space 2} and $T_t = \exp(-tA)$ a semigroup which acts on all $L^p(\Omega),\,1 < p < \infty.$ 
Assume that $T_z$ is analytic on $L^2(\Omega)$ on $z \in \C_+$ and that $T_z$ has an integral kernel which satisfies \eqref{Equ add prop semigroup 1}.
Assume that for $z \in \C_+,$ $\|\exp(-zA)\|_{B(L^2(\Omega))} \lesssim (\cos(\arg z))^{-\frac{d-1}{2}-\beta}( 1 + |\log(\cos(\arg z))|)^2,$ which is the case e.g. when $A$ is self-adjoint.
Then the semigroup $\exp(-zA)$ satisfies on $X = L^p(\Omega)$ for any $1 < p < \infty$ the $R$-bound estimate
\[
R\left( \exp(-e^{i\theta} 2^j t A) :\: j \in \Z \right) \lesssim (\cos(\theta))^{-\frac{d-1}{2}-\beta} (1 + |\log(\cos(\theta))|)^2. 
\]
\end{thm}

For the proof of the theorem, we state two preliminary lemmas.

\begin{lem}\label{Lem Prelim 1}
Let $k_z(x,y)$ be analytic in $z \in \C_+$ and satisfy \eqref{Equ add prop semigroup 1}.
Let $\theta \in (-\frac{\pi}{2},\frac{\pi}{2})$ with $|\theta|$ sufficiently close to $\frac{\pi}{2}$ and $0< t < t_0<\infty.$
Then one has the estimate
\begin{align*} |k_{e^{i\theta}t_0 + t}&(x,y) - k_{e^{i\theta}t_0}(x,y)| \lesssim \\
& \frac{(\cos(\theta))^{-\beta}}{\mu(B(x,t_0))} \frac{\min(\cos(\theta) t_0,t)}{\cos(\theta) t_0} \left[ \left|1 - \frac{\rho(x,y)^2}{t_0^2}\right| + \cos(\theta) \right]^{-\frac{d+1}{2}} + \\
& \begin{cases}
\eqref{Equ Lem Prelim 6} & :\: t_0^2 \geq \rho(x,y)^2\text{ and }\frac{1}{\sqrt{2 - \rho(x,y)^2/t_0^2}-1} \geq 2t_0/t \\
\eqref{Equ Lem Prelim 7} & :\: t_0^2 \geq \rho(x,y)^2\text{ and }\frac{1}{\sqrt{2 - \rho(x,y)^2/t_0^2}-1} \leq 2t_0/t \\
\eqref{Equ Lem Prelim 8} & :\: t_0^2 \leq \rho(x,y)^2 \leq 2 t_0^2 \\
\eqref{Equ Lem Prelim 9} & :\: \rho(x,y)^2 \geq 2 t_0^2,
\end{cases}
\end{align*}
where \eqref{Equ Lem Prelim 6}, \eqref{Equ Lem Prelim 7}, \eqref{Equ Lem Prelim 8}, \eqref{Equ Lem Prelim 9} can be found in the proof of this lemma in Section \ref{Sec Proof Lem Prelim}.
\end{lem}

\begin{lem}\label{Lem Prelim 2}
Let $(\Omega,\mu,\rho)$ be a space of homogeneous type satisfying \eqref{Equ add prop space 1}.
Let $k_z(x,y)$ be analytic in $z \in \C_+$ and satisfy \eqref{Equ add prop semigroup 1}.
Let $t > 0,\:t_0 \in [1,2],\:l \in \Z$ with $l \leq L_{\max} = \max\{j \in \Z:\: 2^j t_0 \leq \diam(\Omega) \} + 1 \leq \infty,\theta \in (-\frac{\pi}{2},\frac{\pi}{2})$ with $|\theta|$ sufficiently close to $\frac{\pi}{2}$ and $x,y \in \Omega$ such that $\rho(x,y) \geq 3t.$
\begin{enumerate}
\item If $|\frac{\rho(x,y)^2}{(2^l t_0)^2} - 1| \leq \cos(\theta),$ then
\begin{align*} \sup_{j \in \Z: 2^j t_0 \geq t} & |k_{e^{i\theta}2^j t_0 + t}(x,y) - k_{e^{i\theta} 2^j t_0}(x,y)| \lesssim \\
& \frac{1}{\mu(B(x,2^l t_0))} \frac{\min(\cos(\theta) 2^l t_0,t)}{\cos(\theta) 2^l t_0} (\cos(\theta))^{-\frac{d+1}{2}-\beta}.
\end{align*}
\item If $\frac{\rho(x,y)^2}{(2^l t_0)^2} \in [1+\cos(\theta),2],$ then with $r = \rho(x,y)^2 / (2^l t_0)^2,$
\begin{align*} \sup_{j \in \Z: 2^j t_0 \geq t} & |k_{e^{i\theta}2^j t_0 + t}(x,y) - k_{e^{i\theta} 2^j t_0}(x,y)| \lesssim \frac{(\cos(\theta))^{-\beta}}{\mu(B(x,2^l t_0))} \times \\
\times & 
\begin{cases}
\left( r - 1 \right)^{-\frac{d+1}{2}} \left( 1 + \log \frac{1 - \sqrt{2 - r}}{\cos(\theta)} \right) \\
+ \left(1 - \sqrt{2 - r} \right)^{-\frac{d+1}{2}} - \left( \frac{t}{2^l t_0} \right)^{-\frac{d+1}{2}}, \\
\text{ if }\cos(\theta) \leq 1 - \sqrt{2 - r} \leq \frac{t}{2^l t_0}; \\
\left( r - 1 \right)^{-\frac{d+1}{2}} \left( 1 + \log \frac{t}{2^l t_0 \cos(\theta)} \right) + |\log \cos(\theta)| ,\\
\text{ if } \cos(\theta) \leq \frac{t}{2^l t_0} \leq 1 - \sqrt{2 - r}; \\
(\cos(\theta))^{-\frac{d+1}{2}}, \text{ if } 1 - \sqrt{2 - r} \leq \cos(\theta) \leq \frac{t}{2^l t_0};\text{ and} \\
\frac{t}{\cos(\theta) 2^l t_0} \left( r - 1 \right)^{-\frac{d+1}{2}}, \text{ if } \cos(\theta) \geq \frac{t}{2^l t_0}.
\end{cases}
\end{align*}
\item If $\frac{\rho(x,y)^2}{(2^l t_0)^2} \in [\frac12,1 - \cos(\theta)],$ then with $r = \rho(x,y)^2/(2^l t_0)^2,$
\begin{align*} \sup_{j \in \Z: 2^j t_0 \geq t} & |k_{e^{i\theta}2^j t_0 + t}(x,y) - k_{e^{i\theta} 2^j t_0}(x,y)| \lesssim \frac{(\cos(\theta))^{-\beta}}{\mu(B(x,2^l t_0))} \times \\
\times &
\begin{cases}
\left(1 - r  \right)^{-\frac{d+1}{2}} \left( 1 + \log \frac{1 -r }{\cos(\theta)} \right) + (1 - r)^{-\frac{d+1}{4}-1} \times \\
\times  \left[ (\sqrt{1-r} - \sqrt{2-r} + 1)^{-\frac{d-1}{2}} - (\sqrt{1-r} - \frac12 \sqrt{2-r} + \frac12 - \frac{\cos(\theta)}{2} )^{-\frac{d-1}{2}}\right] \\
+ |\log(\cos(\theta))|, \text{ if }\cos(\theta) \leq \sqrt{2-r} -1 \leq \frac12 \frac{t}{2^l t_0}; \\
(1-r)^{-\frac{d+1}{2}} \left( 1 + \log \frac{1-r}{\cos(\theta)} \right) + (1-r)^{-\frac{d+1}{4}-1} \times \\
\times \left[ (\sqrt{1-r} - \sqrt{2-r} + 1)^{-\frac{d-1}{2}} - (\sqrt{1-r} - \frac12 \sqrt{2-r} + \frac12 - \frac{\cos(\theta)}{2})^{-\frac{d-1}{2}} \right] \\
+ (\frac{t}{2^l t_0} - \sqrt{2-r} + 1)(1-r)^{-\frac{d+3}{2}} + |\log(\cos(\theta))|, \\
\text{ if }\max(\cos(\theta),\frac12 \frac{t}{2^l t_0}) \leq \sqrt{2-r} -1  \leq \frac{t}{2^l t_0}; \\
(1-r)^{-\frac{d+1}{2}}(1 + \log \frac{r-1}{\cos(\theta)}) + (1-r)^{-\frac{d+1}{4}-1} \times \\
\times \left[(\sqrt{1-r} - \sqrt{2-r} + 1)^{-\frac{d+1}{2}} - (\sqrt{1-r} - \frac12 \sqrt{2-r} + \frac12 - \frac{\cos(\theta)}{2})^{-\frac{d-1}{2}} \right]  \\
+ |\log(\cos(\theta))|, \text{ if }\cos(\theta) \leq \frac{t}{2^l t_0} \leq \sqrt{2-r} - 1; \\
(1-r)^{-\frac{d+1}{2}} + (\frac{t}{2^l t_0} - \sqrt{2-r} + 1)(1-r)^{-\frac{d+3}{2}} + |\log(\cos(\theta))|,\\
\text{ if }\frac12 \frac{t}{2^l t_0} \leq \sqrt{2-R} - 1 \leq \cos(\theta) \leq \frac{t}{2^l t_0}; \\
(1-r)^{-\frac{d+1}{2}} + |\log(\cos(\theta))|,\\
\text{ if }\sqrt{2-r}-1 \leq \min(\cos(\theta),\frac12 \frac{t}{2^l t_0}) \text{ and } \cos(\theta) \leq \frac{t}{2^l t_0};\\
\frac{t}{\cos(\theta) 2^l t_0} (1 - r)^{-\frac{d+1}{2}}\text{ if }\cos(\theta) \geq \frac{t}{2^l t_0}.
\end{cases}
\end{align*}
\end{enumerate}
\end{lem}

The proofs of Lemmas \ref{Lem Prelim 1} and \ref{Lem Prelim 2} are technical and dereferred to Section \ref{Sec Proof Lem Prelim} due to their length.

\begin{proof}[Proof of Theorem \ref{Thm Main}]
Let $\theta \in (-\frac{\pi}{2},\frac{\pi}{2}),\: j \in \Z,\: t_0 \in [1,2]$ and $x,y\in \Omega.$
Write in short $T_j = \exp(-e^{i\theta} 2^j t_0 A).$
Recall that $\{ T_j : \: j \in \Z \}$ is $R$-bounded on $L^p(\Omega)$ with $R$-bound $C < \infty,$ if

\[ \left\| \left( \sum_{j \in F} |T_j f_j|^2 \right)^{\frac12} \right\|_{L^p(\Omega)} \leq C \left\| \left( \sum_{j \in F} |f_j|^2 \right)^{\frac12} \right\|_{L^p(\Omega)}\]

for any finite index set $F \subseteq \Z$ and $f_j \in L^p(\Omega),\:j \in F.$
To prove the theorem, it thus suffices to show that 
\begin{equation}\label{Equ Proof Thm 1}
\left\| T: \begin{cases}
L^p(\Omega,\ell^2(F)) & \to L^p(\Omega,\ell^2(F))\\
(f_j)_{j \in F} & \mapsto (T_j f_j)_{j \in F}
\end{cases} \right\| \lesssim (\cos(\theta))^{-\frac{d-1}{2} - \beta} (1 + |\log(\cos(\theta))|)^2
\end{equation} independently of $F.$
To show this, we apply the boundedness criterion for singular integral operators with non-smooth kernels \cite[Theorem 1]{DuMc} in its vector-valued version \cite[Theorem 2.3]{MoLu}.
Note that the standing assumption in \cite{MoLu} that $\mu(\Omega) = \infty$ is not needed in \cite[Theorem 2.3]{MoLu}.
First note that $L^2(\Omega,\ell^2(F)) = \ell^2(F,L^2(\Omega))$ isometrically, so that for $p = 2,$ $\|T\| \leq \sup_{j \in \Z} \|T_j\|_{B(L^2(\Omega))}$ and that the assumption of the theorem gives that $T : L^2(\Omega,\ell^2(F)) \to L^2(\Omega,\ell^2(F))$ has the norm bound in \eqref{Equ Proof Thm 1}.
To conclude, it suffices to show that $T$ is a vector-valued singular integral operator with non-smooth kernel in the sense of \cite[Definition 2.1]{MoLu}.
We choose the approximation to identity $A_t = A_t' = ((f_j)_{j \in F} \mapsto (\exp(-tA) f_j)_{j \in F}).$
Note that (1.4), (1.5) and (1.6) in \cite{MoLu} are satisfied for this choice due to the Poisson estimate \eqref{Equ add prop semigroup 1} and the volume control \eqref{Equ add prop space 1}.
Note that $T - A_t'T = T - TA_t: (f_j)_{j \in F} \mapsto ((T_j - T_j \exp(-tA)) f_j)_{j \in F},$ that this operator has the $B(\ell^2(F))$-valued kernel with entry $k_{e^{i\theta} 2^j t_0}(x,y) - k_{e^{i\theta} 2^j t_0 + t}(x,y)$ on the diagonal, so that its $B(\ell^2(F))$-norm is controlled by
$\sup_{j \in \Z} |k_{e^{i\theta} 2^j t_0}(x,y) - k_{e^{i\theta} 2^j t_0 + t}(x,y)|.$
By \cite[Definition 2.1 (i) and (ii)]{MoLu}, we are reduced to show that
\begin{align}
\int_{\rho(x,y) \geq 3 t} & \sup_{j \in \Z} |k_{e^{i\theta} 2^j t_0 + t}(x,y) - k_{e^{i\theta} 2^j t_0}(x,y)| d\mu(x) \lesssim \nonumber \\ & (\cos(\theta))^{-\frac{d-1}{2} - \beta} (1 + |\log(\cos(\theta))|)^2 \quad (y \in \Omega,t > 0). \label{Equ Proof Thm 2}
\end{align}
At first, we estimate the above integral by
\[ \int_{\rho(x,y) \geq 3t} \sup_{j:\: 2^j t_0 \leq t} | \ldots | d\mu(x) + \int_{\rho(x,y) \geq 3t} \sup_{j:\: 2^j t_0 \geq t} | \ldots | d\mu(x)\]
and start by estimating the first integral.
Estimate crudely $|k_{e^{i\theta} 2^j t_0 + t}(x,y) - k_{e^{i\theta} 2^j t_0}(x,y)| \leq |k_{e^{i\theta} 2^j t_0 + t}(x,y)| + |k_{e^{i\theta} 2^j t_0}(x,y)|$ and consider both summands separately.
For $\rho(x,y) \geq 3t \geq \frac32 2^j t_0 + \frac32 t,$ we have $|1 + \frac{\rho(x,y)^2}{(e^{i\theta} 2^j t_0 + t)^2}| \cong \frac{\rho(x,y)^2}{|e^{i\theta} 2^j t_0 +t|^2}.$
Thus, 
\begin{align*}
\int_{\rho(x,y) \geq 3t} & \sup_{j:\: 2^j t_0 \leq t} |k_{e^{i\theta} 2^j t_0 + t}(x,y)| d\mu(x) \\
& \lesssim \int_{\rho(x,y) \geq 3t} \sup_{j:\: 2^j t_0 \leq t} \frac{(\cos(\theta))^{-\beta}}{\mu(B(x,|e^{i\theta} 2^j t_0 + t|))} \frac{|e^{i\theta} 2^j t_0 + t|^{d+1}}{\rho(x,y)^{d+1}} d\mu(x) \\
& \lesssim \int_{\rho(x,y) \geq 3t} \sup_{j:\:2^j t_0 \leq t} \frac{(\cos(\theta))^{-\beta}}{\mu(B(x,t))} \frac{t^{d+1}}{\rho(x,y)^{d+1}} d\mu(x) \\
& \lesssim \sum_{k=0}^\infty \int_{B(y,3t\cdot 2^k, 3t \cdot 2^{k+1})} \frac{(\cos(\theta))^{-\beta}}{\mu(B(x,t))} 2^{-k(d+1)} d\mu(x) \\
& \lesssim \sum_{k=0}^\infty \sum_{n=1}^{M_k} \int_{B(y_n^k,t)} \frac{(\cos(\theta))^{-\beta}}{\mu(B(x,t))} 2^{-k(d+1)} d\mu(x) \\
& \lesssim \sum_{k=0}^\infty 2^{kd} 2^{-k(d+1)} (\cos(\theta))^{-\beta} \lesssim (\cos(\theta))^{-\beta},
\end{align*}
where $B(y,3t \cdot 2^0, 3t \cdot 2^{0+1})$ should be replaced by $B(y,3t \cdot 2),$ and we used the two well-known facts on spaces of homogeneous type that $B(y,3t \cdot 2^k, 3t \cdot 2^{k+1})$ can be covered by $M_k \cong 2^{kd}$ balls $B(y_n^k,t)$ and that $\sup_{y \in \Omega} \int_{B(y,t)} \frac{1}{\mu(B(x,t))} d\mu(x) \cong 1.$
Note that if $\mu(\Omega) < \infty,$ then the above sum over $k$ was finite.

Furthermore, for $\rho(x,y) \geq 3t \geq 3 \cdot 2^j t_0,$ we have $|1 + \frac{\rho(x,y)^2}{(e^{i\theta} 2^j t_0)^2}| \cong \frac{\rho(x,y)^2}{(2^j t_0)^2}.$
Thus, similarly to the above calculation, with $M_k \lesssim \left( \frac{t \cdot 2^k}{2^j t_0} \right)^d$ this time,
\begin{align*}
\int_{\rho(x,y) \geq 3t} & \sup_{j:\: 2^j t_0 \leq t} |k_{e^{i\theta} 2^j t_0}(x,y)| d\mu(x) \\
& \lesssim \sum_{j:\: 2^j t_0 \leq t} \int_{\rho(x,y) \geq 3t} \frac{(\cos(\theta))^{-\beta}}{\mu(B(x,2^jt_0))} \frac{(2^j t_0)^{d+1}}{\rho(x,y)^{d+1}} d\mu(x) \\
& \lesssim \sum_{j:\: 2^j t_0 \leq t} \sum_{k=0}^\infty \int_{B(y,3t \cdot 2^k,3t \cdot 2^{k+1})} \frac{(\cos(\theta))^{-\beta}}{\mu(B(x,2^jt_0))} \frac{(2^j t_0)^{d+1}}{(3t \cdot 2^k)^{d+1}} d\mu(x) \\
& \lesssim \sum_{j:\: 2^j t_0 \leq t} \sum_{k=0}^\infty \sum_{n=1}^{M_k} \int_{B(y^k_n,2^j t_0)} \frac{(\cos(\theta))^{-\beta}}{\mu(B(x,2^jt_0))} \frac{(2^j t_0)^{d+1}}{(3t \cdot 2^k)^{d+1}} d\mu(x) \\
& \lesssim \sum_{j:\: 2^j t_0 \leq t} \sum_{k=0}^\infty (\cos(\theta))^{-\beta} \frac{t^d 2^{kd}}{2^{jd} t_0^d} \frac{2^{j(d+1)} t_0^{d+1}}{t^{d+1} 2^{kd} 2^k} \\
& \lesssim \sum_{j:\: 2^j t_0 \leq t} \sum_{k=0}^\infty (\cos(\theta))^{-\beta} t^{-1} 2^{-k} 2^j t_0 
\lesssim (\cos(\theta))^{-\beta}.
\end{align*}

To conclude \eqref{Equ Proof Thm 2}, it now suffices to show
\begin{align}
\int_{\rho(x,y) \geq 3t} & \sup_{j:\:2^j t_0 \geq t} |k_{e^{i\theta} 2^j t_0 + t}(x,y) - k_{e^{i\theta} 2^j t_0}(x,y)| d\mu(x) \lesssim \nonumber \\
& (\cos(\theta))^{-\frac{d-1}{2}-\beta} ( 1 + | \log(\cos(\theta)) | )^2 \quad (y \in \Omega,\: t > 0). \label{Equ Proof Thm 3}
\end{align}
To this end, we use the estimates from Lemma \ref{Lem Prelim 2}.
Therefore, as a first step, we divide the integral in \eqref{Equ Proof Thm 3} into
\begin{align*} \int_{\rho(x,y) \geq 3t} & \leq \sum_{l=L_{\min}}^{L_{\max}} \int_{|\rho(x,y)^2/(2^l t_0)^2 - 1| \leq \cos(\theta)}
+ \sum_{l=L_{\min}}^{L_{\max}} \int_{\rho(x,y)^2/(2^l t_0)^2 \in [1 + \cos(\theta),2]} \\
& + \sum_{l=L_{\min}}^{L_{\max}} \int_{\rho(x,y)^2/(2^l t_0)^2 \in [\frac12, 1 - \cos(\theta)]} =: \sum_{l=L_{\min}}^{L_{\max}} I_1+I_2+I_3,
\end{align*}
where $L_{\min} = \min\{ l \in \Z :\: 2^l t_0 \geq \frac{3}{2\sqrt{1 - \cos(\theta)}} t \}$ and $L_{\max} = \max\{ l \in \Z:\: 2^l t_0 \leq \diam(\Omega) \} + 1.$
For $I_1,$ we have by Lemma \ref{Lem Prelim 2} 1. and the volume control \eqref{Equ add prop space 2},
\begin{align*}
& \int_{|\rho(x,y)^2/(2^l t_0)^2 - 1| \leq \cos(\theta)} \sup_{j :\: 2^j t_0 \geq t} |k_{e^{i\theta} 2^j t_0 + t}(x,y) - k_{e^{i\theta} 2^j t_0}(x,y)| d\mu(x) \\
& \lesssim \int_{\ldots} \frac{1}{\mu(B(x,2^l t_0))} \frac{\min(\cos(\theta)2^l t_0,t)}{\cos(\theta) 2^l t_0} (\cos(\theta))^{-\frac{d+1}{2} - \beta} d\mu(x) \\
& \lesssim 2^{ld} \cos(\theta) 2^{-ld} \frac{\min(\cos(\theta)2^l t_0,t)}{\cos(\theta) 2^l t_0} (\cos(\theta))^{-\frac{d+1}{2} - \beta} \\
& \cong  \frac{\min(\cos(\theta)2^l t_0,t)}{\cos(\theta) 2^l t_0} (\cos(\theta))^{-\frac{d-1}{2} - \beta}.
\end{align*}
Now summing over $l,$ we obtain
\begin{align*}
\sum_{l=L_{\min}}^{L_{\max}} & \frac{\min(\cos(\theta)2^l t_0,t)}{\cos(\theta) 2^l t_0} (\cos(\theta))^{-\frac{d-1}{2} - \beta} \\
& \lesssim (\cos(\theta))^{-\frac{d-1}{2} - \beta} \left( \int_1^{(\cos(\theta))^{-1}} \frac{dx}{x} + \int_{(\cos(\theta))^{-1}}^\infty (\cos(\theta))^{-1} x^{-1} \frac{dx}{x} \right) \\
& \cong (\cos(\theta))^{-\frac{d-1}{2} - \beta} ( 1 + |\log(\cos(\theta))| ).
\end{align*}

Now for $I_2.$
We distinguish the cases I) $(\cos(\theta) 2^l t_0 \leq t)$ and II) $(\cos(\theta) 2^l t_0 \geq t).$
In case I), we decompose $I_2$ into three integrals according to Lemma \ref{Lem Prelim 2} 2.
We write $r = \rho(x,y)^2 / (2^l t_0)^2.$
\begin{align*}
I_2 & \leq \int_{\cos(\theta) \leq 1 - \sqrt{2 - r} \leq \frac{t}{2^l t_0}, 1 + \cos(\theta) \leq r \leq 2} + \int_{\cos(\theta) \leq \frac{t}{2^l t_0} \leq 1 - \sqrt{2 - r}} + \int_{1 - \sqrt{2 - r} \leq \cos(\theta) \leq \frac{t}{2^l t_0}} \\
& =: I_2^1 + I_2^2 + I_2^3.
\end{align*}
Then by the volume condition \eqref{Equ add prop space 2} and Lemma \ref{Lem Prelim 2} 2.,
\begin{align*}
I_2^1 & \lesssim \int_{\cos(\theta) \leq 1 - \sqrt{2-r} \leq \frac{t}{2^l t_0},1 + \cos(\theta) \leq r \leq 2} \!\!\!\!\!\!\!\!\!\!\!\!\!\!\!\!\!\!\!\! (\cos(\theta))^{-\beta}(r-1)^{-\frac{d+1}{2}} \left( 1 + \log \frac{r-1}{\cos(\theta)} \right) r^{\frac{d}{2}} \frac{dr}{r} \\
& \lesssim (\cos(\theta))^{-\beta} \int_{1 + \cos(\theta)}^2 (r-1)^{-\frac{d+1}{2}} \left( 1 + \log \frac{r-1}{\cos(\theta)} \right) dr \\
& \lesssim (\cos(\theta))^{-\frac{d-1}{2}-\beta} \left( 1 + \max_{\cos(\theta) \leq r \leq 1} \log \frac{r}{\cos(\theta)} \right) \\
& \cong (\cos(\theta))^{-\frac{d-1}{2} - \beta} (1 + |\log(\cos \theta)|).
\end{align*}
Further,
\begin{align*}
I_2^2 & \lesssim  (\cos(\theta))^{-\beta} \int_{1 + \cos(\theta) \leq r \leq 2,\frac{t}{2^l t_0} \leq 1 - \sqrt{2-r}}
\left((r-1)^{-\frac{d+1}{2}} 
\left( 1 + \log \frac{t}{2^l t_0 \cos(\theta)} \right) \right. \\
& \left. + |\log(\cos(\theta))| \right) r^{\frac{d}{2}} \frac{dr}{r} \\
& \lesssim (\cos(\theta))^{-\beta} \left( \int_{1+ \cos(\theta)}^2 (r-1)^{-\frac{d+1}{2}} dr ( 1 + \log \frac{t}{2^l t_0} + |\log(\cos(\theta))|) + |\log(\cos(\theta))| \right) \\
& \lesssim (\cos(\theta))^{-\frac{d-1}{2} - \beta}( 1 + |\log(\cos(\theta))| ),
\end{align*}
due to $\frac{t}{2^l t_0} \leq 1 - \sqrt{2 - r}\leq r-1 \leq 1$ in this case.
Finally,
\begin{align*}
I_2^3 & \lesssim (\cos(\theta))^{-\beta} \int_{1 - \sqrt{2 - r} \leq \cos(\theta),1 + \cos(\theta) \leq r \leq 2} (\cos(\theta))^{-\frac{d+1}{2}}  r^{\frac{d}{2}} \frac{dr}{r} \\
& \lesssim (\cos(\theta))^{- \frac{d+1}{2} - \beta} \int_{1 + \cos(\theta)}^{1 + 2 \cos(\theta)} r^{\frac{d}{2}} \frac{dr}{r} \\
& \cong (\cos(\theta))^{- \frac{d-1}{2} - \beta},
\end{align*}
where we have used that $\frac12 (r-1) \leq 1 - \sqrt{2-r},$ and thus $r-1 \leq 2 (1 - \sqrt{2-r}) \leq 2 \cos(\theta).$
Now summing up the three estimates, we get $I_2 \lesssim (\cos(\theta))^{-\frac{d-1}{2} - \beta} ( 1 + |\log(\cos(\theta))| ),$
and thus, as for $\sum_{l = L_{\min}}^{L_{\max}} I_1,$ we deduce $\sum_{l = L_{\min}}^{L_{\max}} I_2 \lesssim  (\cos(\theta))^{-\frac{d-1}{2} - \beta} ( 1 + |\log(\cos(\theta))| )^2.$
In case II), we have by the volume condition \eqref{Equ add prop space 2} and Lemma \ref{Lem Prelim 2} 2.,
\begin{align*}
I_2 & \cong (\cos(\theta))^{-\beta} \frac{t}{\cos(\theta) 2^l t_0} \int_{1 + \cos(\theta)}^2 ( r - 1)^{-\frac{d+1}{2}} r^{\frac{d}{2}} \frac{dr}{r} \\
& \cong (\cos(\theta))^{-\beta} \frac{t}{\cos(\theta) 2^l t_0} \left[-(r-1)^{-\frac{d-1}{2}} \right]_{1 + \cos(\theta)}^2 
\lesssim (\cos(\theta))^{-\frac{d-1}{2} - \beta} \frac{t}{\cos(\theta) 2^l t_0}.
\end{align*}
Now summing over those $l$ with $\cos(\theta) 2^l t_0 \geq t,$ we obtain
$ \sum_{l:\:\cos(\theta) 2^l t_0 \geq t} I_2 \lesssim (\cos(\theta))^{-\frac{d-1}{2}-\beta} \int_0^1 x \frac{dx}{x} =  (\cos(\theta))^{-\frac{d-1}{2}-\beta}.$\\

It remains to estimate $I_3,$ for which we use part 3. of Lemma \ref{Lem Prelim 2}.
Again we distinguish the cases I) $(\cos(\theta) 2^l t_0 \leq t)$ and II) $(\cos(\theta) 2^l t_0 \geq t)$.
In case I), we decompose $I_3 \lesssim \int \frac{(\cos(\theta))^{-\beta}}{\mu(B(x,2^l t_0))} ( 1 + |\log(\cos(\theta))| ) d\mu(x) + \int \text{remainder } d\mu(x).$
The first term is estimated by
\begin{align*}
& \int_{\rho(x,y)^2/(2^l t_0)^2 \in [\frac12, 1 - \cos(\theta)]}  \frac{(\cos(\theta))^{-\beta}}{\mu(B(x,2^l t_0))} (1 + | \log(\cos(\theta)) |) d\mu(x) \\
& \lesssim \int_{\frac12}^{1 - \cos(\theta)} (\cos(\theta))^{-\beta} ( 1 + | \log(\cos(\theta)) |) r^{\frac{d}{2}} \frac{dr}{r} \\
& \cong (\cos(\theta))^{-\beta} ( 1 + | \log(\cos(\theta)) | ).
\end{align*}
The remainder is decomposed into the following five integrals corresponding to the first five cases in Lemma \ref{Lem Prelim 2} 3.
We write again $r = \rho(x,y)^2 / ( 2^l t_0 )^2.$
\begin{align*}
& \int_{\frac12 \leq r \leq 1 - \cos(\theta),\cos(\theta) \leq \sqrt{2 - r} -1 \leq \frac12 \frac{t}{2^l t_0}}
+ \int_{\frac12 \leq r \leq 1 - \cos(\theta), \max(\cos(\theta),\frac12 \frac{t}{2^l t_0}) \leq \sqrt{2-r}-1 \leq \frac{t}{2^l t_0}} \\
& + \int_{\frac12 \leq r \leq 1 - \cos(\theta), \cos(\theta) \leq \frac{t}{2^l t_0} \leq \sqrt{2-r} -1}
+ \int_{\frac12 \leq r \leq 1 - \cos(\theta), \frac12 \frac{t}{2^l t_0} \leq \sqrt{2-r} -1 \leq \cos(\theta) \leq \frac{t}{2^l t_0}}\\
& + \int_{\frac12 \leq r \leq 1 - \cos(\theta), \sqrt{2-r} - 1 \leq \min(\cos(\theta),\frac12 \frac{t}{2^l t_0}}
=: I_3^1 + I_3^2 + I_3^3 + I_3^4 + I_3^5.
\end{align*}
We estimate the five integrals separately.
For $I_3^1,$ we have
\begin{align*}
& (\cos(\theta))^\beta I_3^1 \lesssim \int_{\frac12 \leq r \leq 1 - \cos(\theta), \cos(\theta) \leq 1 - r \leq \frac12 \frac{1}{\sqrt{2} - 1} \frac{t}{2^l t_0}} (1-r)^{-\frac{d+1}{2}} ( 1 + \log \frac{1-r}{\cos(\theta)}) r^{\frac{d}{2}} \frac{dr}{r} \\
& + \int_{\frac12 \leq r \leq 1 - \cos(\theta),\cos(\theta) \leq \sqrt{2-r}-1 \leq \frac12 \frac{t}{2^l t_0}} (1 -r)^{-\frac{d+1}{4}-1} \times \\
& \times (1-r) (\sqrt{1-r} - (\sqrt{2-r} -1))^{-\frac{d+1}{2}} r^{\frac{d}{2}} \frac{dr}{r} \\
& \lesssim (\cos(\theta))^{-\frac{d-1}{2}} ( 1 + |\log(\cos(\theta))| ) + \int_{\frac12}^{1 - \cos(\theta)} (1-r)^{-\frac{d+1}{4}-1 +1 - \frac{d+1}{4}} dr \\
& \lesssim  (\cos(\theta))^{-\frac{d-1}{2}} ( 1 + |\log(\cos(\theta))| ) + (\cos(\theta))^{-\frac{d-1}{2}},
\end{align*}
where we have used the mean value theorem to estimate
\begin{align*} 
& (\sqrt{1-r} - (\sqrt{2-r} -1))^{-\frac{d-1}{2}} - (\sqrt{1-r} - \frac12 ( \sqrt{2-r} - 1) - \frac{\cos(\theta)}{2})^{-\frac{d-1}{2}} \lesssim \\
& (1-r) (\sqrt{1-r} - (\sqrt{2-r} -1))^{-\frac{d+1}{2}}.
\end{align*}
For $I_3^2,$ we have
\begin{align*}
& (\cos(\theta))^\beta I_3^2 \lesssim \int_{\frac12 \leq r \leq 1 - (\cos(\theta)),\max(\cos(\theta),\frac12 \frac{t}{2^l t_0}) \leq \sqrt{2-r} - 1 \leq \frac{t}{2^l t_0}} \\
& \left\{ 1^{st}\text{ term }+ 2^{nd} \text{ term }+(\frac{t}{2^l t_0} - \sqrt{2-r} +1)(1-r)^{-\frac{d+3}{2}} \right\} r^{\frac{d}{2}} \frac{dr}{r},
\end{align*}
where the $1^{st}$ and $2^{nd}$ term are given in Lemma \ref{Lem Prelim 2} 3., and can be controlled as in $I_3^1$ by $(\cos(\theta))^{-\frac{d-1}{2}} ( 1 + |\log(\cos(\theta))| ).$
The third term, we estimate by
\begin{align*}
& \lesssim \int_{\frac12 \leq r \leq 1 - \cos(\theta),\max(\cos(\theta),\frac12 \frac{t}{2^l t_0} \leq \sqrt{2-r} -1 \leq \frac{t}{2^l t_0}} ( \frac{t}{2^l t_0} - (\sqrt{2-r} -1))(1-r)^{-\frac{d+3}{2}} dr \\
& \lesssim \int_{\frac12 \leq r \leq 1 - \cos(\theta)} (1-r)(1-r)^{-\frac{d+3}{2}} dr \lesssim (\cos(\theta))^{-\frac{d-1}{2}}.
\end{align*}
For $I_3^3,$ we apply exactly the same estimate as for $I_3^1,$ to get $(\cos(\theta))^\beta I_3^3 \lesssim (\cos(\theta))^{-\frac{d-1}{2}} ( 1 + | \log(\cos(\theta)) | ),$ too.
For $I_3^4,$ we have again with $1^{st},2^{nd}$ and $3^{rd}$ term given in the estimate in Lemma \ref{Lem Prelim 2} 3.,
\begin{align*}
& (\cos(\theta))^{\beta} I_3^4 \lesssim \int_{\frac12 \leq r \leq 1 - \cos(\theta), \frac12 \frac{t}{2^l t_0} \leq \sqrt{2-r} - 1 \leq \cos(\theta) \leq \frac{t}{2^l t_0} } \\
&  \left\{ 1^{st} \text{ term } + 2^{nd} \text{ term } + 3^{rd} \text{ term } \right\} r^{\frac{d}{2}} \frac{dr}{r}.
\end{align*}
The $1^{st}$ and $2^{nd}$ term can be controlled as in $I_3^1,$ whereas the $3^{rd}$ term can be controlled as in $I_3^2.$
We get $(\cos(\theta))^\beta I_3^4 \lesssim (\cos(\theta))^{-\frac{d-1}{2}} ( 1+ | \log(\cos(\theta)) | ).$

Eventually, we estimate $I_3^5$ by
\begin{align*}
& (\cos(\theta))^\beta I_3^5 \lesssim \int_{\frac12 \leq r \leq 1 - \cos(\theta),\sqrt{2-r}-1 \leq \min(\cos(\theta),\frac12 \frac{t}{2^l t_0}) }
(1-r)^{-\frac{d+1}{2}} r^{\frac{d}{2}} \frac{dr}{r} \\
& \lesssim (\cos(\theta))^{-\frac{d-1}{2}}.
\end{align*}

Now summing up we obtain as for $I_2,$ $\sum_{l= L_{\min}}^{L_{\max}} I_3 \leq \sum_{l= L_{\min}}^{L_{\max}} I_3^1 + I_3^2 + I_3^3 + I_3^4 + I_3^5 \lesssim (\cos(\theta))^{-\frac{d-1}{2} - \beta} ( 1 + |\log(\cos(\theta))|)^2.$

In case II), we have
\begin{align*}
& (\cos(\theta))^\beta I_3 \lesssim \int_{\frac12 \leq r \leq 1 - \cos(\theta)} \frac{t}{\cos(\theta) 2^l t_0} ( 1 - r)^{-\frac{d+1}{2}} r^{\frac{d}{2}} \frac{dr}{r} \\
& \lesssim \frac{t}{\cos(\theta) 2^l t_0} (\cos(\theta))^{-\frac{d-1}{2}}.
\end{align*}
Summing over $l$ with $\cos(\theta) 2^l t_0 \geq t,$ we obtain as for $I_2,$ $\sum_{l:\: \cos(\theta) 2^l t_0 \geq t} I_3 \lesssim (\cos(\theta))^{-\frac{d-1}{2} - \beta}.$

We have shown that
\[ \sum_{l = L_{\min}}^{L_{\max}} I_1 + I_2 + I_3 \lesssim (\cos(\theta))^{-\frac{d-1}{2} - \beta} ( 1 + | \log(\cos(\theta)) | )^2, \]
and thus \eqref{Equ Proof Thm 3} and hence the theorem follows.
\end{proof}

\begin{rem}
We have proved in Theorem \ref{Thm Main} that the semigroup is $R$-bounded when taking dyadic arguments of the form $e^{i\theta} 2^j t_0,\: t_0 \in [1,2]$ fixed.
It is unclear whether one gets also an $R$-bounded set when this argument is replaced by a continuous variant, i.e. if $R(\exp(-e^{i\theta} t_0 A):\: t_0 > 0) \lesssim (\cos(\theta))^{-\frac{d-1}{2} - \beta} (1 + | \log(\cos(\theta)) | )^2$ holds under the same hypotheses as in Theorem \ref{Thm Main}, and at least the above proof does not seem to work.
By the method in \cite[Last part of 2.16 Example]{KuWe}, which passes from dyadic arguments to continuous ones, one only gets a cruder estimate for the continuous variant involving an additional factor $(\cos(\theta))^{-2}.$
\end{rem}

\begin{cor}\label{Cor Ha calculus}
Let $(\Omega,\mu,\rho)$ be a space of homogeneous type of either finite or infinite measure, satisfying \eqref{Equ add prop space 1} and \eqref{Equ add prop space 2}.
Let $T_t = \exp(-tA)$ a semigroup which acts on all $L^p(\Omega),\,1 < p < \infty.$ 
Assume that $T_z$ is analytic on $L^2(\Omega)$ on $z \in \C_+$ and that $T_z$ has an integral kernel $k_z(x,y)$ which satisfies the Poisson estimate
\[ | k_{re^{i\theta}}(x,y)| \leq C (\cos(\theta))^{-\beta} \frac{1}{\mu(B(x,r))} \frac{1}{\left|1 + \frac{\rho(x,y)^2}{(re^{i\theta})^2} \right|^{\frac{d+1}{2}}} \] 
for any $x,y \in \Omega, r > 0, \theta \in (-\frac{\pi}{2},\frac{\pi}{2})$ and some $C,\beta \geq 0.$
Assume moreover that $A$ has a bounded $\HI(\Sigma_\omega)$ calculus on $L^2(\Omega)$ for some $\omega \in (0,\pi),$ and that for $r > 0, \theta \in (-\frac{\pi}{2},\frac{\pi}{2}),$
\[ \|\exp(-re^{i\theta}A)\|_{B(L^2(\Omega))} \leq C (\cos(\theta))^{-\frac{d-1}{2}-\beta}( 1 + |\log(\cos(\theta))| )^2,\] 
both of which are satisfied e.g. when $A$ is self-adjoint.
Then $A$ has a bounded $\Ha$ calculus on $L^p(\Omega)$ for any $1<p<\infty$ and $\alpha > \frac{d}{2}.$
Moreover this calculus is an $R$-bounded mapping, i.e.
\[ R(f(A):\:\|f\|_{\Ha} \leq 1) < \infty.\]
\end{cor}

\begin{proof}
This follows immediately from Propositions \ref{Prop Prelim Ha calculus}, \ref{Prop A has HI calculus} and Theorem \ref{Thm Main}.
\end{proof}

\section{Examples}\label{Sec 4 Examples}

In this section, we want to give some applications of Theorem \ref{Thm Main} and Corollary \ref{Cor Ha calculus}.
Note first that in the most classical example, namely $A = (-\Delta)^{\frac12}$ and $\exp(-zA)$ the Poisson semigroup on $L^p(\R^d)$ for some $1 < p < \infty$ and $d \in \N,$ Corollary \ref{Cor Ha calculus} gives the sharp order of derivation of the classical H\"ormander multiplier theorem and strengthens it in that it includes the $R$-boundedness of spectral multipliers whose associated functions have bounded $\Ha$ norm.

For a generalization, we consider the situation in \cite{MMMM}.
Let $M$ be a positive integer.
Consider the constant coefficient second order, $M\times M$ system, differential operator
\[ L u  = \sum_{\gamma=1}^M \sum_{r,s=1}^{d+1} \left( \partial_r ( a_{rs}^{\alpha \gamma} \partial_s u_\beta ) \right)_{1 \leq \alpha \leq M}, \]
where $a_{rs}^{\alpha \beta}$ are real coefficients for $r,s=1,\ldots,d+1$ and $\alpha,\gamma = 1,\ldots,M.$
Here, $u$ is a function defined on the upper half space $\R^{d+1}_+ = \R^d \times [0,\infty).$
Further, we assume as in \cite{MMMM} the ellipticity condition
\[ \sum_{\alpha,\gamma = 1}^M \sum_{r,s = 1}^{d+1} \Re \left[ a^{\alpha \gamma}_{rs} \xi_r \xi_s \overline{\eta_\alpha} \eta_\beta \right] \geq \kappa_0 |\xi|^2 |\eta|^2 \]
for every $(\xi_r)_{1 \leq r \leq d+1} \in \R^{d+1},$ $(\eta_\alpha)_{1 \leq \alpha \leq M} \in \C^M$ and some $\kappa_0 > 0.$
Then in \cite{MMMM} the following Dirichlet problem on $\R^{d+1}_+$ is considered:
\[ \begin{cases}
Lu = 0 \text{ in }\R^{d+1}_+ \\
u|_{\partial\R^{d+1}_+}^{n.t.} = f \in L^p(\R^d;\C^M),
\end{cases} \]
where $\partial\R^{d+1}_+ = \R^d \times \{ 0 \},$ $n.t.$ means non-tangential trace of $u,$ and $f$ is a given function in $L^p(\R^d;\C^M),\:1 < p < \infty.$
If $\mathfrak{A}^{dis}_L \not= \emptyset,$ a certain condition, see \cite[(3.12)]{MMMM}, which will be satisfied in our example, then this problem is well-posed in $L^p(\R^d;\C^M)$ \cite[Theorem 4.1]{MMMM}, so it possesses a unique solution $u.$
As moreover the coefficients defining $L$ are constant, we have $\partial_r[u(\cdot+(0,t))]|_{\cdot = x} = (\partial_r u)(x + (0,t)),$ so that the expression $T_t f(x) := u(x,t),\: x \in \R^d,\:t \geq 0$ defines a semigroup on $L^p(\R^d;\C^M).$
In the sequel, we are interested in the H\"ormander functional calculus of the negative generator of that semigroup.
Note that for some cases, this semigroup is given by a convolution kernel.
We now restrict to the following specific example.

\paragraph{Lam\'e system of elasticity.}
Assume that $M = d + 1$ above.
The so-called Lam\'e operator in $\R^{d+1}$ has the form
\begin{equation}\label{Equ Lame operator}
Lu = \mu \Delta u + (\lambda + \mu) \nabla \dive u,\quad u = (u_1,\ldots,u_{d+1}),
\end{equation}
where the constants $\lambda,\mu \in \R$ (typically called Lam\'e moduli) are assumed to satisfy $\mu > 0$ and $2\mu + \lambda > 0.$
Then according to \cite[Theorem 5.2]{MMMM}, $T_tf(x) = u(x,t)$ is given by
\begin{align}
& (T_tf)_\alpha(x) = \frac{4\mu}{3\mu + \lambda}\frac{1}{\omega_d} \int_{\R^d} \frac{t}{(|x-y|^2 + t^2)^{\frac{d+1}{2}}}f_\alpha(y) dy \nonumber \\
& + \frac{\mu + \lambda}{3 \mu + \lambda} \frac{2(d+1)}{\omega_d} \sum_{\gamma=1}^{d+1}\int_{\R^d} \frac{t(x-y,t))_\alpha (x-y,t)_\gamma}{(|x-y|^2 + t^2)^{\frac{d+3}{2}}}f_\gamma(y) dy \quad (\alpha = 1,\ldots,d+1), \label{Equ Lame sgr}
\end{align}
where $\omega_d$ is the area of the unit sphere $S^d$ in $\R^{d+1},$ and $(x-y,t)_\alpha = x_\alpha - y_\alpha$ if $\alpha \in \{ 1 , \ldots , d \}$ and $(x-y,t)_{d+1} = t.$

\begin{prop}\label{Prop Lame}
The semigroup in \eqref{Equ Lame sgr} is strongly continuous on $L^p(\R^d;\C^{d+1})$ for $1 < p < \infty$, has an analytic extension for $\Re z > 0$ and if $k_z(x,y) = (k_{z;\alpha\gamma})_{\alpha,\gamma=1}^{d+1}(x,y)$ denotes its $(d+1)\times(d+1)$ matrix valued integral kernel, then each of its components $k_{z,\alpha\gamma}(x,y)$  satisfies the Poisson estimate \eqref{Equ add prop semigroup 1} with $\beta = 1.$
Further, the negative generator $A$ of the semigroup has an $\HI$ calculus on $L^2(\R^d;\C^{d+1})$ and 
\begin{align*}
\|\exp(-e^{i\theta}tA)\|_{B(L^2(\R^d;\C^{d+1}))} & \lesssim (\cos(\theta))^{-1}\max(1+ |\log(\cos(\theta))|,(\cos(\theta))^{-\frac{d-1}{2}}) \\
& \lesssim (\cos(\theta))^{-\frac{d-1}{2}-\beta} ( 1+ |\log(\cos(\theta))|)^2
\end{align*}
for $t > 0$ and $|\theta| < \frac{\pi}{2}.$
\end{prop}

\begin{proof}
For the strong continuity of the semigroup, we show first that $\sup_{t > 0}$ $\|T_t\|_{L^p(\R^d;\C^{d+1}) \to L^p(\R;\C^{d+1})} < \infty.$
As $T_t$ is given by a linear combination of convolutions, we have $\|T_t\|_{p \to p} \lesssim \int_{\R^d} \frac{t}{(|y|^2 + t^2)^{\frac{d+1}{2}}} dy + \int_{\R^d} t \frac{|y|^2 + t^2}{(|y|^2 + t^2)^{\frac{d+3}{2}}} dy = 2 \int_{\R^d} \frac{1}{(|y/t|^2 + 1)^{\frac{d+1}{2}}} \frac{dy}{t^d},$
which is clearly bounded independently of $t > 0.$
Thus it suffices to show the strong continuity $\|T_t f - f\|_p \to 0$ as $t \to 0$ for $f$ belonging to the dense subset of continuous functions with compact support.
In the following calculation, we use that the kernel $k_t(x,y)$ of $T_t$ satisfies $\int_{\R} k_t(x,y) dy = 1_{M_{d+1}},$ since $T_t1_{\ell^2_{d+1}} = 1_{\ell^2_{d+1}}$ is the unique solution of the Dirichlet problem with constant initial value, and we use Jensen's inequality.
\begin{align*}
& \|T_t f - f\|_p^p \cong \sum_{\alpha = 1}^{d+1} \int_{\R^d} \left| \int_{\R^d} c_1 \frac{t}{(|y|^2 + t^2)^{\frac{d+1}{2}}}f_\alpha(x-y) dy
\right. \\
& \left. + \sum_{\gamma = 1}^{d+1} \int_{\R^d} c_2 t \frac{(y,t)_\alpha (y,t)_\gamma}{(|y|^2 + t^2)^{\frac{d+3}{2}}} f_\gamma(x-y) dy - f_\alpha(x)  \right|^p dx \\
& = \sum_{\alpha = 1}^{d+1} \int_{\R^d} \left| \int_{\R^d} c_1 \frac{t}{(|y|^2 + t^2)^{\frac{d+1}{2}}}[f_\alpha(x-y)-f_\alpha(x)] dy \right. \\
& \left. + \sum_{\gamma =1}^{d+1} \int_{\R^d} c_2 t \frac{(y,t)_\alpha (y,t)_\gamma}{(|y|^2 + t^2)^{\frac{d+3}{2}}} [f_\gamma(x-y) dy - f_\gamma(x)] dy \right|^p dx \\
& \lesssim \sum_{\alpha = 1}^{d+1} \int_{\R^d} \left\{ c_1 t \int_{\R^d} \frac{1}{(|y|^2 + t^2)^{\frac{d+1}{2}}} |f_\alpha(x-y)-f_\alpha(x)|^p dy \right. \\
& \left.+ c_2 t \sum_{\gamma = 1}^{d+1} \int_{\R^d} \frac{|(y,t)_\alpha (y,t)_\gamma|}{(|y|^2 + t^2)^{\frac{d+3}{2}}} |f_\gamma(x-y) - f_\gamma(x)|^p dy \right\} dx \\
& = \sum_{\alpha = 1}^{d+1} \int_{\R^d} \left\{ c_1 \int_{\R^d} \frac{1}{(|y|^2 + 1)^{\frac{d+1}{2}}} |f_\alpha(x-ty)-f_\alpha(x)|^p dy \right. \\
& \left. + c_2 \sum_{\gamma = 1}^{d+1} \int_{\R^d} \frac{|(y,1)_\alpha (y,1)_\gamma|}{(|y|^2 + 1)^{\frac{d+3}{2}}} |f_\gamma(x- ty) - f_\gamma(x)|^p dy \right\} dx \\
& \cong \sum_{\alpha} \int_{|x| \leq R} \int_{\R^d} \ldots dy + \sum_{\gamma} \int_{\R^d} \ldots dy dx + \int_{|x| \geq R} \int_{\R^d} \ldots dy + \sum_{\gamma} \int_{\R^d} \ldots dy dx.
\end{align*}
Now choose first $\epsilon > 0$ and $R >> 1$ sufficiently large.
Then there exists $t_0 << 1$ such that $|f_\alpha(x-ty)-f_\alpha(x)|^p < \frac{\epsilon}{\mu(B(0,R))}$ for $t \leq t_0$ and $\alpha \in \{1,\ldots,d+1\}.$
Thus the integral over $|x| \leq R$ is $\lesssim \epsilon.$
Let $\supp f \subset B(0,r).$
Then the integral over $|x| \geq R$ is $\lesssim \int_{|y| \geq 1/t_0 (R - r)} \frac{1}{(|y|^2+1)^{\frac{d+1}{2}}} 2^p \|f\|_p^p dy \to 0$ as $R \to \infty.$
We have proved the strong continuity of the semigroup.

For the analyticity, extend the definition of $T_t$ from \eqref{Equ Lame sgr} by replacing $t > 0$ by a complex $z$ with $\Re z > 0$ everywhere.
Clearly, \eqref{Equ Lame sgr} is an analytic expression in $t,$ so that we get an analytic family $T_z.$
The claimed $B(L^2(\R^d;\C^{d+1}))$ estimate of $T_z$ follows from the integral kernel estimate of $\frac{z}{(|\cdot|^2 + z^2)^{\frac{d+1}{2}}},$ which is shown in \cite[p.~348]{GaPy} and gives the estimate for the first expression in \eqref{Equ Lame sgr}, and also for the second expression by the pointwise estimate $\frac{|z(y,z)_\alpha(y,z)_\gamma|}{\left| |y|^2 + z^2 \right|^{\frac{d+3}{2}}} \leq (\cos(\arg z))^{-1}\frac{|z|}{\left| |y|^2 + z^2 \right|^{\frac{d+1}{2}}},$ which is shown in the following.

Now for the claimed Poisson estimate \eqref{Equ add prop semigroup 1} with $\beta = 1.$
Fix some $\alpha,\gamma \in \{1,\ldots,d+1\}.$
We have by \eqref{Equ Lame sgr}, 
\begin{equation}\label{Equ Lame kernel}
k_{z,\alpha\gamma}(x,y) = c_1 \delta_{\alpha = \gamma} \frac{z}{(|x-y|^2 + z^2)^{\frac{d+1}{2}}} + c_2 \frac{z(x-y,z)_\alpha (x-y,z)_\gamma}{(|x-y|^2 + z^2)^{\frac{d+3}{2}}}.
\end{equation}
Clearly, the first summand admits the complex Poisson estimate \eqref{Equ add prop semigroup 1} even with $\beta = 0.$
For the second summand, it clearly suffices to show that $|(x,z)_\alpha (x,z)_\gamma| \lesssim (\cos(\theta))^{-1} \left| |x|^2 + z^2 \right|,$ where $\arg z = \theta.$
We distinguish the four cases $\alpha,\gamma \leq d;\: \alpha,\gamma = d+1;\: \alpha \leq d,\: \gamma = d + 1$ and $\alpha = d + 1,\: \gamma \leq d.$
Suppose first $\alpha,\gamma \leq d.$
If $|x|^2 \geq 2 |z|^2,$ then $|x_\alpha x_\gamma| \leq \frac12 (x_\alpha^2 + x_\gamma^2) \leq \frac12 |x|^2 \lesssim \left| |x|^2 + z^2 \right|.$
If $|x|^2 \leq \frac12 |z|^2,$ then $|x_\alpha x_\gamma| \leq \frac12 |x|^2 \leq \frac14 |z|^2 \leq \frac12 (|z|^2 - |x|^2) \leq \frac12 \left| z^2 + |x|^2 \right|.$
If $\frac12 |z|^2 \leq |x|^2 \leq 2 |z|^2,$ then 
\begin{align*} 
|x_\alpha x_\gamma| \cos(\theta) & \lesssim \cos(\theta) |z|^2 \leq |z|^2 \cos(\theta) + \left| |z|^2 (\cos^2(\theta) - \sin^2(\theta) + |x|^2 \right| \\
& \cong |2 (\Re z)(\Im z)| + \left|(\Re z)^2 - (\Im z)^2 +|x|^2 \right| \\
&= |\Im(z^2)| + \left|\Re(z^2) + |x|^2\right| \cong \left|z^2 + |x|^2 \right|.
\end{align*}
The other three cases can be treated in the same manner.

We now show that the negative generator $A$ of $T_t$ has a bounded $\HI$ calculus on $L^2(\R^d;\C^{d+1}).$
According to \cite[Theorem 2.4]{CDMY}, it suffices to show that both
\begin{equation}\label{Equ sgr scalar 2}
\int_0^\infty t \| A e^{-tA} f\|_{2}^2 dt \lesssim \|f\|_2^2 \text{ and }\int_0^\infty t \|A' e^{-tA'} f\|_{2}^2 dt \lesssim \|f\|_2^2
\end{equation}
hold.
We have with $F$ denoting the Fourier transform, using the Plancherel formula and Fubini,
\begin{align*}
\int_0^\infty t \|Ae^{-tA}f\|_{L^2(\R^d;\C^{d+1})}^2 dt & = \int_0^\infty \sum_{\alpha = 1}^{d+1} t \left\|\frac{\partial}{\partial t} (e^{-tA} f)_\alpha\right\|_{L^2(\R^d)}^2 dt \\
& = \int_0^\infty \sum_{\alpha = 1}^{d+1} t \left\| \sum_{\gamma = 1}^{d+1} \frac{\partial}{\partial t} k_{t,\alpha\gamma} \ast f_\gamma \right\|_2^2 dt \\
& = \int_0^\infty \sum_{\alpha = 1}^{d+1} t \left\| \sum_{\gamma = 1}^{d+1} \frac{\partial}{\partial t} F[ k_{t,\alpha\gamma}] \cdot F[f_\gamma] \right\|_2^2 dt \\
& \lesssim \sum_{\alpha,\gamma = 1}^{d+1} \int_{\R^d} |F[f_\gamma](\xi)|^2 \int_0^\infty t \left| \frac{\partial}{\partial t} F[k_{t,\alpha\gamma}](\xi) \right|^2 dt d\xi.
\end{align*}
For the first estimate in \eqref{Equ sgr scalar 2}, it thus suffices to show that
\begin{equation}\label{Equ Lame sgr 2}
\int_0^\infty t \left| \frac{\partial}{\partial t} F[k_{t,\alpha\gamma}](\xi) \right|^2 dt \lesssim 1
\end{equation}
independently of $\xi \in \R^d.$
We decompose $k_{t,\alpha\gamma} =k_{t,\alpha\gamma}^1 + k_{t,\alpha\gamma}^2$ into the two summands according to \eqref{Equ Lame kernel}.
For the first summand, we have $F[k_{t,\alpha\gamma}^1(\xi)] = c_1 \delta_{\alpha = \gamma} e^{-t|\xi|}.$
Thus,
\begin{align*}
& \int_0^\infty t \left| \frac{\partial}{\partial t} F[k_{t,\alpha\gamma}^1](\xi)\right|^2 dt = c_1^2 \delta_{\alpha = \gamma} \int_0^\infty t |\xi|^2 e^{-2t|\xi|} dt \\
& = c_1^2 \delta_{\alpha = \gamma} \int_0^\infty t e^{-2t} dt < \infty.
\end{align*}
Next we claim that 
\begin{equation}\label{Equ Lame sgr 3}
F \left[x \mapsto \frac{t}{(|x|^2 + t^2)^{\frac{d+3}{2}}}\right](\xi) = c_d \frac{|\xi|}{t} e^{-t|\xi|}(1 + \frac{1}{t|\xi|}),
\end{equation}
from which we shall easily deduce $F[k_{t,\alpha\gamma}^2](\xi).$
First note that the function which we Fourier transform in \eqref{Equ Lame sgr 3} is invariant under rotation, so also its Fourier transform is.
Thus we can assume that $\xi = (\xi_1,0,\ldots,0)$ with $\xi_1 \geq 0$ and we get
(l.h.s. of \eqref{Equ Lame sgr 3}) $= c_1 \int_{\R} \ldots \int_{\R} \frac{t}{(x_1^2 + x_2^2 + \ldots + x_d^2 + t^2)^{\frac{d+3}{2}}} e^{-ix_1 \xi_1} dx_2 \ldots dx_d dx_1.$
Note that $\int_{\R} \frac{t}{(x^2 + s^2)^a} dx = c_a \frac{t}{(s^2)^{a-\frac12}}$ for any $a > \frac12,$ so that by induction, we get
by \cite[p.~202]{Ober},
(l.h.s. of \eqref{Equ Lame sgr 3}) $= c_d' \int_{\R} \frac{t}{(x_1^2 + t^2)^2} e^{-ix_1\xi_1} dx_1 = c_d' 2 t \sqrt{\pi} \Gamma(2)^{-1} \left( \frac{\xi_1}{2t} \right)^{\frac32} K_{\frac32}(t \xi_1),$
where $K_{\frac32}(x) = \sqrt{\frac{\pi}{2x}} e^{-x} ( 1 + \frac1x )$ is a modified Bessel function.
This immediately gives \eqref{Equ Lame sgr 3}.
Now we calculate $\frac{\partial}{\partial t}F[k_{t,\alpha\gamma}^2](\xi)$ and distinguish the five cases $\alpha = \gamma = d+1;\:\alpha \leq d$ and $\gamma = d+1;\: \alpha = d+1 $ and $\gamma \leq d;\:\alpha,\gamma \leq d$ and $\alpha \neq \gamma;$ and finally $\alpha = \gamma \leq d.$

If $\alpha = \gamma = d+1,$ then $\frac{\partial}{\partial t} F[k_{t,\alpha\gamma}^2](\xi)= c_d \frac{\partial}{\partial t} \left(t|\xi| e^{-t|\xi|}(1 + \frac{1}{t|\xi|}) \right) = c_d e^{-t|\xi|}(-t|\xi|^2).$
Thus, $\int_0^\infty t \left| \frac{\partial}{\partial t} F[k_{t,\alpha\gamma}^2](\xi) \right|^2 dt = c_d \int_0^\infty t^2 |\xi|^2 e^{-2t |\xi|}(-t|\xi|)^2 \frac{dt}{t} < \infty$ independently of $|\xi|.$
If $\alpha \leq d$ and $\gamma = d+1,$ then $\frac{\partial}{\partial t} F[ k_{t,\alpha\gamma}^2 ](\xi) = c_d i \frac{\partial}{\partial t}\frac{\partial}{\partial \xi_\alpha}\left( |\xi| e^{-t|\xi|}(1 + \frac{1}{t|\xi|} ) \right) = c_d i e^{-t|\xi|}(t \xi_\alpha |\xi| - \xi_\alpha).$
Thus, $\int_0^\infty t \left| \frac{\partial}{\partial t} F[k_{t,\alpha\gamma}^2](\xi)\right|^2 dt = c_d \int_0^\infty e^{-t} (t^2 \frac{\xi_\alpha}{|\xi|} - t \frac{\xi_\alpha}{|\xi|} )^2 \frac{dt}{t} < \infty$ independently of $\xi.$
If $\alpha = d+1$ and $\gamma \leq d,$ one obtains the same as above, roles of $\alpha$ and $\gamma$ interchanged.
If $\alpha,\gamma \leq d$ and $\alpha \neq \gamma,$ we have
$\frac{\partial}{\partial t} F[k_{t,\alpha \gamma}^2](\xi) = - c_d \frac{\partial}{\partial t} \frac{\partial}{\partial \xi_\alpha} \frac{\partial}{\partial \xi_\gamma} (e^{-t|\xi|}(\frac{|\xi|}{t} + \frac{1}{t^2})) = - c_d e^{-t|\xi|} (-t\xi_\alpha \xi_\gamma + \frac{\xi_\alpha \xi_\gamma}{|\xi|}).$
Thus, $\int_0^\infty t \left| \frac{\partial}{\partial t} F[k_{t,\alpha\gamma}^2](\xi)\right|^2 dt = c_d \int_0^\infty e^{-2t} ( -t^2 \frac{\xi_\alpha\xi\gamma}{|\xi|^2} + t \frac{\xi_\alpha\xi_\gamma}{|\xi|^2})^2 \frac{dt}{t} < \infty$ independently of $\xi.$
If finally $\alpha = \gamma \leq d,$ then $\frac{\partial}{\partial t} F[k_{t,\alpha\gamma}^2](\xi) = - c_d e^{-t|\xi|} ( -t \xi_\alpha^2 + |\xi| + \frac{\xi_\alpha^2}{\xi} ),$ and again $\int_0^\infty t \left| \frac{\partial}{\partial t} F[k_{t,\alpha\gamma}^2](\xi)\right|^2 dt < \infty$ independently of $\xi.$

We have shown the first inequality in \eqref{Equ sgr scalar 2}.
The second one follows by the same proof, there are only signs without further impact which change.
This shows that $A$ has an $\HI$ calculus.
\end{proof}

\begin{cor}\label{Cor Lame}
Let $A$ be the negative generator of the Lam\'e semigroup given in \eqref{Equ Lame sgr} on $L^p(\R^d;\C^{d+1})$ for some $1 < p < \infty.$
Then $A$ has an $\Ha$ calculus for any $\alpha > \frac{d}{2} + 1.$
Moreover, $\{f(A):\: \|f\|_{\Ha} \leq 1 \}$ is $R$-bounded on $L^p(\R^d;\C^{d+1})$ for these $\alpha.$
\end{cor}

\begin{proof}
The corollary follows immediately from Propositions \ref{Prop Prelim Ha calculus} and \ref{Prop Lame} once we show that we can apply Proposition \ref{Prop A has HI calculus} and that $\{T(e^{i\theta} 2^k t_0): \: k \in \Z \}$ is $R$-bounded with 
\begin{equation}\label{Equ R-bound Lame}
R\left(\left\{ T(e^{i\theta} 2^k t_0) :\: k \in \Z \right\} \right) \lesssim (\cos(\theta))^{-\frac{d-1}{2} - 1} ( 1 + |\log(\cos(\theta))| )^2.
\end{equation}
For the application of Proposition \ref{Prop A has HI calculus}, note that a careful inspection of the proof of \cite[Theorem 3.1]{DuRo} shows that a version in $L^p(\R^d;\C^{d+1})$ also holds, more precisely that:
if $T_z$ is an analytic semigroup on $L^2(\R^d;\C^{d+1})$ for $\Re z > 0$ such that its negative generator $A$ has a bounded $\HI$ calculus on $L^2(\R^d;\C^{d+1})$ and each of the $(d+1) \times (d+1)$ components of its matrix valued integral kernel $k_{z,\alpha\gamma}(x,y)$ satisfies \eqref{Equ add prop semigroup 1}, then $A$ has a bounded $\HI$ calculus on $L^p(\R^d;\C^{d+1})$ for $1 < p < \infty.$
Note that we have shown the $\HI$ calculus on $L^2(\R^d;\C^{d+1})$ and the bounds \eqref{Equ add prop semigroup 1} for $k_{z,\alpha\gamma}(x,y)$ in Proposition \ref{Prop Lame}.
It remains to show \eqref{Equ R-bound Lame}, which would follow immediately from Theorem \ref{Thm Main}, weren't it for the vector valued character.
We give the adaption details now.
First note that $\left( \E \left\| \sum_{k=1}^n \epsilon_k f_k \right\|_{L^p(\R^d;\C^{d+1})}^2 \right)^{\frac12} \cong \sum_{\alpha= 1}^{d+1} \left\| \left( \sum_{k=1}^n |f_{\alpha,k}|^2 \right)^{\frac12} \right\|_{L^p(\R^d)} \cong \left\| \left( \sum_{\alpha=1}^{d+1} \sum_{k=1}^n |f_{\alpha,k}|^2 \right)^{\frac12} \right\|_{L^p(\R^d)}.$
Thus,

$\{T(e^{i\theta}2^kt_0):\: k\in\Z \}$ is $R$-bounded on $L^p(\R^d;\C^{d+1})$ iff
\[\left\| \left( \sum_{\alpha=1}^{d+1} \sum_{k} |(T(e^{i\theta}2^{k} t_0 A)f_k)_\alpha|^2 \right)^{\frac12} \right\|_p \lesssim \left\| \left( \sum_{\alpha=1}^{d+1} \sum_k |f_{\alpha,k}|^2 \right)^{\frac12} \right\|_p,\]
iff $T \in B(L^p(\R^d;\ell^2_{\{1,\ldots,d+1\}\times \N})),$ where $T(f_{\alpha,j})_{\alpha,j} = (T(e^{i\theta}2^{j} t_0)f_j)_\alpha.$
In view of \cite[Definition 2.1]{MoLu}, choose now the approximation of identity $A_t(f_{\alpha,j})_{\alpha,j} = (T_tf_j)_\alpha.$
Then $TA_t(f_{\alpha,j}) = T(T_tf_j) = (T(e^{i\theta}2^{j}t_0)T_t f_j) = A_t T(f_{\alpha,j}).$
Further, the kernel of $A_t,$ called $a_t$ satisfies
\[ A_t(f_{\alpha,j}) = \int_{\R^d} \sum_{\gamma = 1}^{d+1} k_{t,\alpha\gamma}(x-y)f_{\gamma,j}(y) dy \]
and
\begin{align*}
\|a_t(x-y)\|_{B(\ell^2_{\{1,\ldots,d+1\} \times \N})} & = \sup_{w_{\alpha,j}:\sum_{\alpha,j}|w_{\alpha,j}|^2 \leq 1} \left(\sum_{\alpha,j} \left| \sum_{\gamma = 1}^{d+1} k_{t,\alpha\gamma}(x-y)w_{\gamma,j}\right|^2  \right)^{\frac12} \\
& \lesssim \sup_{\alpha,\gamma = 1}^{d+1} |k_{t,\alpha\gamma}(x-y)| \lesssim \frac{1}{\mu(B(x,t))} \frac{1}{\left|1 + \frac{|x-y|^2}{t^2}\right|^{\frac{d+1}{2}}},
\end{align*}
so that $A_t$ is indeed an approximation of identity in the sense of \cite[Definition 1.1]{MoLu}.
If $k_T(x,y)$ and $k_{TA_t}(x,y)$ denote the $B(\ell^2_{\{1,\ldots,d+1\}\times \N})$ valued kernels of $T$ and $TA_t$, then we have in view of the application of \cite[Theorem 2.3]{MoLu},
\begin{align*}
& \int_{|x-y|\geq 3t} \|k_T(x,y) - k_{TA_t}(x,y)\|_{B(\ell^2_{\{1,\ldots,d+1\}\times \N})} \lesssim \\
& \sup_{\alpha,\gamma=1}^{d+1} \int_{|x-y| \geq 3t} \sup_{j \in \Z} |k_{e^{i\theta}2^j t_0,\alpha\gamma}(x-y) - k_{e^{i\theta}2^j t_0 +t,\alpha\gamma}(x-y)| d\mu(x).
\end{align*}
Apply now the Poisson estimate and analyticity of $k_{z,\alpha\gamma}(x-y)$ in $z,$ for fixed $\alpha$ and $\gamma,$ exactly as in the proof of Theorem \ref{Thm Main}, to deduce \eqref{Equ R-bound Lame} as wanted.
\end{proof}

\begin{rem}
Let us compare Corollary \ref{Cor Lame} with known H\"ormander functional calculi for ``Lam\'e operators'' in the literature.
Let $L$ denote the operator as in \eqref{Equ Lame operator}, i.e. $Lu = \mu \Delta u + (\lambda + \mu)\nabla \dive u,$ where $u : \Omega \to \R^{d+1}$ and $\Omega \subset \R^d$ is an open subset satisfying the interior ball condition, i.e. there exists a positive constant $c$ such taht for all $x \in \Omega$ and all $r \in(0,\frac12 \diam(\Omega)),\:\mu(B(x,r)) \geq c r^d.$
Let further $A$ denote the negative generator of the semigroup in \eqref{Equ Lame sgr}, so that if $\Omega = \R^d,$ for functions $u : \R^{d+1}_+ \to \C^{d+1}$ and $f \in L^p(\R^d;\C^{d+1}),$ we have $Lu = 0,\; u|_{\partial\R^{d+1}_+}^{n.t.} = f$ if and only if $u(x,t) = e^{-tA}f(x).$
If $\Omega$ is bounded, the operator $-L$ on $L^p(\Omega;\C^{d+1})$ for $p \in (q_\Omega',q_\Omega)$ where $q_\Omega > 2$ is some constant, has a $\Ha$ calculus for $\alpha > d | \frac1p - \frac12 | + \frac12$ \cite[Theorem 5.1]{KuU}.
In contrast, for $A$ we get a H\"ormander functional calculus on the full range $p \in (1,\infty),$ but with a worse derivation exponent $\alpha > \frac{d}{2} + 1.$
Note that $A$ is \emph{not} self-adjoint on $L^2(\R^d;\C^{d+1}),$ in contrast to $L$ on $L^2(\Omega;\C^{d+1}),$ so that $A$ and $L$ are of a quite different nature.
\end{rem}

\paragraph{Dirichlet-to-Neumann operator}
Another application of Theorem \ref{Thm Main} and Corollary \ref{Cor Ha calculus} is the following operator from \cite{OtE}.
There the authors suppose that $\Omega$ is the smooth boundary of an open connected subset $\tilde\Omega$ of $\R^{d+1}$
and $A$ is the Dirchlet-to-Neumann operator defined as follows:
Given $\phi \in L^2(\Omega)$ solve the Dirichlet problem
\begin{align*}
 \Delta u & = 0 \text{ weakly on }\tilde\Omega \\
 u|_\Omega & = \phi
\end{align*}
with $u \in W^{1}_2(\tilde\Omega).$
If $u$ has a weak normal derivative $\frac{\partial u}{\partial \nu}$ in $L^2(\Omega),$
then $\phi \in D(A)$ and $A\phi = \frac{\partial u}{\partial \nu}.$
This operator is a pseudodifferential operator, self-adjoint on $L^2(\Omega).$
Then in \cite{OtE} it is shown that the semigroup satisfies the complex Poisson estimate:
\begin{align*}
|k_z(x,y)| & \leq C (\cos (\theta))^{-2(d-1)d} \frac{\min(|z|,1)^{-d}}{\left( 1 + \frac{|x-y|}{|z|}\right)^{d+1}} \\
& \lesssim (\cos(\theta))^{-2(d-1)d} \frac{1}{\mu(B(x,|z|))} \frac{1}{\left|1 + \frac{|x-y|^2}{z^2} \right|^{\frac{d+1}{2}}}
\end{align*}
for all $x,y \in \Omega$ and $\Re z > 0,$ where $\theta = \arg z.$
Further a $\Ha$ calculus for $A$ with $\alpha > \frac{d}{2}$ is derived in \cite[Section 7]{OtE}.
Suppose that $\Omega$ satisfies \eqref{Equ add prop space 1} and \eqref{Equ add prop space 2}.
Then since $A$ is self-adjoint on $L^2(\Omega),$
we can apply Corollary \ref{Cor Ha calculus} to deduce that $A$ has an $R$-bounded $\Ha$ calculus for $\alpha > \frac{d}{2} + 2(d-1)d.$
Note that our derivation order in this functional calculus is worse than the one obtained in \cite[Section 7]{OtE}, but since it is $R$-bounded, it contains square function estimates like
\[ \left\| \left( \sum_{j=1}^n |g_j(A) f_j|^2 \right)^{\frac12} \right\|_p \lesssim \max_{j=1}^n \|g_j\|_{\Ha} \left\| \left( \sum_{j=1}^n |f_j|^2 \right)^{\frac12} \right\|_p .\]

\paragraph{Pseudodifferential operators on compact manifolds without boundary}

Let $\Omega$ be a compact closed (i.e. without boundary) $d$-dimensional Riemannian $C^\infty$-manifold and $A$ a classical, self-adjoint, strongly elliptic pseudodifferential operator on $\Omega$ of order $1$ such that $\gamma(A) = \inf\{\Re z:\: z \in \sigma(A)\} \geq 0.$
Then according to \cite[Theorem 3.14]{GiGr}, the semigroup generated by $-A$ has an integral kernel satisfying
\[ |k_z(x,y)| \lesssim (\cos(\theta))^{-\beta} e^{-\gamma(A) \Re z} \frac{|z|}{\rho(x,y) + |z|} ( (\rho(x,y) + |z|)^{-d} + 1) \quad (\Re z > 0) \]
with $\theta = \arg z$ and $\beta =\frac72 d + 11.$
If $\Omega$ satisfies \eqref{Equ add prop space 1}, then this estimate readily gives \eqref{Equ add prop semigroup 1}.
Thus if $\Omega$ also satisfies \eqref{Equ add prop space 2}, we can appeal to Corollary \ref{Cor Ha calculus} and deduce that $A$ has an $R$-bounded $\Ha$ calculus for $\alpha > \frac{d}{2} + \frac72 d + 11$ on $L^p(\Omega),\: 1 < p < \infty.$
The order of this calculus is worse than what is known in the literature for this kind of operator ($\alpha > \frac{d}{2},$ see \cite{SeSo}), but at least, our result includes the $R$-boundedness of the calculus.

\section{Proofs of Lemmas \ref{Lem Prelim 1} and \ref{Lem Prelim 2}}\label{Sec Proof Lem Prelim}

\begin{proof}[Proof of Lemma \ref{Lem Prelim 1}]
Since $k_z(x,y)$ is analytic in $z,$ one has
\begin{align*}
|k_{e^{i\theta}t_0 + t}(x,y) - k_{e^{i\theta}t_0}(x,y)| & \leq \int_0^t \left| \frac{\partial}{\partial s}k_{e^{i\theta}t_0 + s}(x,y) \right| ds \\
& \leq \int_0^t \frac{1}{2\pi} \int_{\Gamma_{e^{i\theta}t_0+s}} \left|\frac{k_z(x,y)}{(z-e^{i\theta}t_0 - s)^2}\right| dz ds,
\end{align*}
where $\Gamma_{e^{i\theta}t_0+s}$ is the contour of the circle centered at $e^{i\theta}t_0 + s$ with radius $r (\cos(\theta)t_0+s),$
where $r>0$ is some small constant determined later on in this proof.
The rest of the proof is the rather long task to exploit \eqref{Equ add prop semigroup 1} in the above double integral.
Let $z \in \Gamma_{e^{i\theta}t_0 + s}.$
Then $\cos(\arg z) \cong \min(|\frac{\Re z}{\Im z}|,1) \cong \min(\frac{\cos(\theta)t_0 + s}{\sin(\theta)t_0 +s}, 1) \cong \cos(\arg(e^{i\theta}t_0 +s)) \geq \cos(\theta).$
Also, $|z| \cong |e^{i\theta}t_0 + s|.$
We divide the integral over $s$ in 
\[\int_0^t \ldots ds = \int_0^{\min(\cos(\theta) t_0,t)} \ldots ds + \int_{\min(\cos(\theta) t_0,t)}^t \ldots ds.\]

\noindent \underline{1st case:} $\cos(\theta) t_0 \geq s.$

We show that

\begin{equation}\label{Equ Lem Prelim 4}
\left| 1 + \frac{\rho(x,y)^2}{[e^{i\theta} t_0 + s + r e^{i\phi} (\cos(\theta) t_0 + s)]^2} \right| \cong \left|1 + \frac{\rho(x,y)^2}{(e^{i\theta}t_0)^2}\right|.
\end{equation}

First note that \eqref{Equ Lem Prelim 4} is equivalent to $|w_l| \cong |w_r|,$ where
$w_l = [e^{i\theta} t_0 + s + r e^{i\phi} (\cos(\theta) t_0 + s)]^2 + \rho(x,y)^2$ and $w_r = (e^{i\theta} t_0)^2 + \rho(x,y)^2.$
We have $\Re w_r = \cos(2\theta) t_0^2 + \rho(x,y)^2 = -t_0^2 + \rho(x,y)^2 + o(\frac{\pi}{2} - |\theta|) t_0^2$ and $|\Im w_r| \cong \cos(\theta) t_0^2.$
This gives $|w_r| \cong |\Re w_r| + |\Im w_r| \cong |\rho(x,y)^2 - t_0^2| + \cos(\theta) t_0^2.$
On the other hand, $|\Im w_l| = 2 (\cos(\theta) t_0 + s)(1 + r \cos(\phi))|\sin(\theta)t_0 + r \sin(\phi)(\cos(\theta)t_0 + s)| \cong \cos(\theta) t_0^2$ and $\Re w_l = \rho(x,y)^2 - t_0^2 + t_0^2 (1 - \sin^2(\theta)) -r^2 \sin^2(\phi) (\cos(\theta) t_0 + s)^2 - 2r \sin(\theta) t_0 \sin(\phi) (\cos(\theta) t_0 + s) + (\cos(\theta) t_0 + s)^2(1 + r \cos(\phi))^2.$
This gives 
\begin{align*}
|\Re w_l| & \geq |\rho(x,y)^2 - t_0^2| - \left[ (\frac{\pi}{2} - |\theta|)^2 + o((\frac{\pi}{2} - |\theta|)^2) \right] t_0^2 \\
&  - 4r^2t_0^2(\frac{\pi}{2} - |\theta|)^2 ( 1 + o(1)) -2r\cdot2\cdot(\frac{\pi}{2} - |\theta|) t_0^2 \\
& \geq |\rho(x,y)^2 - t_0^2| - c_1 (\frac{\pi}{2} - |\theta|) t_0^2,
\end{align*}
where $c_1 << 1$ is some small constant if $r$ is sufficiently small and $|\theta|$ is sufficiently close to $\frac{\pi}{2}.$
Similarly, $|\Re w_l| \leq |\rho(x,y)^2 - t_0^2| + c_1 (\frac{\pi}{2} - |\theta|) t_0^2.$
Thus, $|w_l| \cong |\Re w_l| + |\Im w_l| \cong |\rho(x,y)^2 - t_0^2| + \cos(\theta) t_0^2 \cong |w_r|,$ and \eqref{Equ Lem Prelim 4} follows.
Now we can estimate

\begin{align}
& \int_0^{\min(\cos(\theta) t_0,t)} \left| \frac{\partial}{\partial s} k_{e^{i\theta}t_0 + s}(x,y) \right| ds \lesssim \nonumber \\
& \frac{(\cos(\theta))^{-\beta}}{\mu(B(x,t_0))} \int_0^{\min(\cos(\theta) t_0,t)} \frac{1}{\cos(\theta) t_0} \left\{ |\rho(x,y)^2 - t_0^2| + \cos(\theta) t_0^2 \right\}^{-\frac{d+1}{2}} t_0^{d+1} ds \nonumber \\
& = \frac{(\cos(\theta))^{-\beta}}{\mu(B(x,t_0))} \frac{\min(\cos(\theta)t_0,t)}{\cos(\theta) t_0} \left\{ |1 - \frac{\rho(x,y)^2}{t_0^2}| + \cos(\theta) \right\}^{-\frac{d+1}{2}}. \label{Equ Lem Prelim 5}
\end{align}


\noindent \underline{2nd case:} $\cos(\theta) t_0 \leq s.$

We show that
\begin{equation}\label{Equ Lem Prelim 2}
\left| 1 + \frac{\rho(x,y)^2}{[e^{i\theta} t_0 + s + r e^{i\phi}(\cos(\theta) t_0 + s)]^2} \right| \gtrsim \left| 1 + \frac{\rho(x,y)^2}{(s+it_0)^2} \right|.
\end{equation}
Put $w_l = [e^{i\theta} t_0 + s + r e^{i\phi}(\cos(\theta)t_0 + s)]^2 + \rho(x,y)^2$ and $w_r = [s+it_0]^2 + \rho(x,y)^2,$
so that \eqref{Equ Lem Prelim 2} $\Longleftrightarrow |w_l| \gtrsim |w_r|.$
We have $\Re w_r = s^2 - t_0^2 + \rho(x,y)^2$ and $\Im w_r = 2s t_0.$
On the other hand, $\Re w_l = \left[ \cos(\theta) t_0 + s + r \cos(\phi) (\cos(\theta) t_0 + s)\right]^2 - \left[\sin(\theta) t_0 + r \sin(\phi) (\cos(\theta) t_0 +s)\right]^2 + \rho(x,y)^2,$ and
\begin{align*}
|\Im w_l| & = 2 \left|\left[\cos(\theta) t_0 + s + r \cos(\phi)(\cos(\theta) t_0 + s)\right]\right. \times \\
& \times \left.\left[ \sin(\theta) t_0 + r \sin(\phi) (\cos(\theta) t_0 + s)\right]\right| \cong s t_0,
\end{align*}
if $r < \frac14.$
This gives
\begin{align*}
\Re w_l & = \cos^2(\theta) t_0^2 + s^2 + r^2 \cos^2(\phi)(\cos(\theta) t_0 + s)^2 + 2s \cos(\theta) t_0 \\
& + 2 \cdot r (\cos(\theta) t_0 + s) \cos(\phi) (\cos(\theta) t_0 + s) - \sin^2(\theta) t_0^2 \\
& - r^2 \sin^2(\phi) (\cos(\theta) t_0 + s) - 2 \cdot r \sin(\theta) t_0 \sin(\phi) (\cos(\theta) t_0 + s) + \rho(x,y)^2.
\end{align*}
Now we distinguish the two cases that $st_0$ is bigger or smaller than $|s^2 - t_0^2 + \rho(x,y)^2|.$
In the first case, we have $|w_l| \geq |\Im w_l| \cong st_0 \gtrsim |\Re w_r| + |\Im w_r| \cong |w_r|.$
In the second case, we have with constants $c_1<1$ arbitrarily close to $1$ and $c_2>0$ arbitrarily close to $0$ if $r$ is sufficiently small, and $c_3>0$ such that $(\cos(\theta))^{-1}\leq c_3 (\frac{\pi}{2} - |\theta|)^{-1}$
\begin{align*}
|\Re w_l| & \geq |c_1 s^2 + \cos(2\theta) t_0^2 + \rho(x,y)^2| - c_2 t_0 s \\
& \geq |s^2 - t_0^2 + \rho(x,y)^2| - (1-c_1)s^2 - (1+\cos(2\theta)) t_0^2 - c_2 t_0 s \\
& \geq |s^2 - t_0^2 + \rho(x,y)^2| - (1-c_1)s^2 - (2(\frac{\pi}{2}-|\theta|)^2 + o((\frac{\pi}{2}-|\theta|)^3)) \times \\
& \times (\cos(\theta))^{-1}c_3s t_0 -c_2 t_0s \\
& \geq |s^2 - t_0^2 + \rho(x,y)^2| - st_0 [(1-c_1) + 2c_3(\frac{\pi}{2} - |\theta|) + o((\frac{\pi}{2}-|\theta|)^2) + c_2] \\
& \gtrsim |s^2 - t_0^2 + \rho(x,y)^2|.
\end{align*}
Thus, $|w_l| \geq |\Re w_l| \gtrsim |s^2 - t_0^2 + \rho(x,y)^2| + st_0 \gtrsim |\Re w_r| + |\Im w_r| \cong |w_r|.$
This shows \eqref{Equ Lem Prelim 2}.

It follows that if $\cos(\theta)t_0 \leq t,$ then 
\begin{equation}\label{Equ Lem Prelim 3}
\int_{\cos(\theta) t_0}^t \left| \frac{\partial}{\partial s}k_{e^{i\theta}t_0 + s}(x,y) \right| ds \lesssim \frac{(\cos(\theta))^{-\beta}}{\mu(B(x,t_0))} \int_{\cos(\theta) t_0}^t \frac{1}{s} \left| 1 + \frac{\rho(x,y)^2}{(s+it_0)^2} \right|^{-\frac{d+1}{2}} ds.
\end{equation}
We have 
\begin{align*}
\left| 1 + \frac{\rho(x,y)^2}{(s+i t_0)^2} \right|^{-\frac{d+1}{2}} & \cong |s + i t_0|^{d+1} |(s+it_0)^2 + \rho(x,y)^2|^{-\frac{d+1}{2}} \\
& \cong t_0^{d+1} \left[ \max(|s^2  - t_0^2 + \rho(x,y)^2|,2st_0)\right]^{-\frac{d+1}{2}}.
\end{align*}
Next we determine the value of the above maximum of two terms.
A simple calculation shows that the two terms are the same iff $s$ takes one of the four values $s_{\pm,\pm} =\pm t_0 \pm \sqrt{2 t_0^2 - \rho(x,y)^2}.$
We distinguish three cases i) - iii). \\

\noindent\underline{Case i)} $t_0^2 \geq \rho(x,y)^2.$

Then out of $s_{\pm,\pm},$ only $s_{-,+}= -t_0 + \sqrt{2t_0^2 - \rho(x,y)^2}$ lies in $[0,t_0].$
If $s \leq s_{-,+},$ then $|s^2 - t_0^2 + \rho(x,y)^2| \geq 2st_0$ and if $s \geq s_{-,+},$ then $|s^2 - t_0^2 + \rho(x,y)^2| \leq 2st_0.$
Now divide the integral in \eqref{Equ Lem Prelim 3} accordingly.
We get
\begin{align}
& \int_{\cos(\theta) t_0}^t  \left| \frac{\partial}{\partial s} k_{e^{i\theta}t_0 +s}(x,y)\right| ds \lesssim \nonumber \\
& \frac{(\cos(\theta))^{-\beta}}{\mu(B(x,t_0))} \int_{\cos(\theta) t_0}^{t_0 (\sqrt{2-\rho(x,y)^2/t_0^2}-1)} \frac{1}{s} t_0^{d+1} |s^2 - t_0^2 + \rho(x,y)^2|^{-\frac{d+1}{2}} ds  \nonumber \\
& + \frac{(\cos(\theta))^{-\beta}}{\mu(B(x,t_0))} \int_{t_0(\sqrt{2-\rho(x,y)^2/t_0^2}-1)}^t \frac{1}{s} t_0^{d+1} (st_0)^{-\frac{d+1}{2}} ds \nonumber \\
& \cong  \frac{(\cos(\theta))^{-\beta}}{\mu(B(x,t_0))} \left\{ \int_{\cos(\theta) t_0}^{t_0 (\sqrt{2-\rho(x,y)^2/t_0^2}-1)} \frac{1}{s} [(\sqrt{t_0^2 - \rho(x,y)^2} - s) \times \right. \nonumber \\
& \left. \times (\sqrt{t_0^2 - \rho(x,y)^2} + s)]^{-\frac{d+1}{2}} t_0^{d+1} ds + 
\left(\sqrt{2 - \frac{\rho(x,y)^2}{t_0^2}} - 1 \right)^{-\frac{d+1}{2}} - (t_0/t)^{\frac{d+1}{2}} \right\}. \label{Equ Lem Prelim 6}
\end{align}

The last integral above, with lower and upper bound abbreviated by $a$ and $b,$ can be further estimated by

\begin{align*}
& \int_a^b \frac{1}{s} t_0^{d+1} (t_0^2 - \rho(x,y)^2)^{-\frac{d+1}{4}} \left[ \frac{1}{(\sqrt{t_0^2 - \rho(x,y)^2}-s)^{\frac{d+1}{2}}}
+ \frac{1}{(\sqrt{t_0^2 - \rho(x,y)^2} +s )^{\frac{d+1}{2}}} \right] ds \\
& \lesssim \int_{a/t_0}^{b/t_0} \left( 1 - \frac{\rho(x,y)^2}{t_0^2} \right)^{-\frac{d+1}{4}}  \frac{1}{(\sqrt{1 - \frac{\rho(x,y)^2}{t_0^2}} - s)^{\frac{d+1}{2}}} \frac{ds}{s}.
\end{align*}

\noindent\underline{Case ia)} $\frac{1}{\sqrt{2 - \rho(x,y)^2/t_0^2} - 1} \geq 2 t_0/t.$

Then the term after the integral in \eqref{Equ Lem Prelim 6} can be simplified to
\begin{align*}
\left(\sqrt{2 - \frac{\rho(x,y)^2}{t_0^2}}-1 \right)^{-\frac{d+1}{2}} \cong \left(1 - \frac{\rho(x,y)^2}{t_0^2}\right)^{-\frac{d+1}{2}} \cong \left(1-\frac{\rho(x,y)}{t_0}\right)^{-\frac{d+1}{2}}.
\end{align*}

\noindent\underline{Case ib)} $\frac{1}{\sqrt{2 - \rho(x,y)^2/t_0^2} - 1} \leq 2 t_0/t.$

Then we obtain

\begin{align}
& \int_{\cos(\theta)t_0}^t \left| \frac{\partial}{\partial s}k_{e^{i\theta}t_0 + s}(x,y) \right| ds \lesssim \nonumber \\
& \frac{(\cos(\theta))^{-\beta}}{\mu(B(x,t_0))} \left\{ ( 1 - \frac{\rho(x,y)^2}{t_0^2} )^{-\frac{d+1}{4}} \left[ \left(\sqrt{1 - \frac{\rho(x,y)^2}{t_0^2}}\right)^{-\frac{d+1}{2}} \times \right.\right. \nonumber \\
& \times \log\left(\frac{\frac12[\sqrt{2-\rho(x,y)^2/t_0^2}-1 + \cos(\theta)]}{\cos(\theta)}\right)
+ [\sqrt{2 - \rho(x,y)^2/t_0^2} -1 + \cos(\theta)]^{-1} \times \nonumber \\
& \left[\left(\sqrt{1-\frac{\rho(x,y)^2}{t_0^2}} - \sqrt{2 - \frac{\rho(x,y)^2}{t_0^2}} + 1\right)^{-\frac{d-1}{2}} \right. \nonumber \\
& \left.\left.- \left(\sqrt{1 - \frac{\rho(x,y)^2}{t_0^2}} - \frac12 \sqrt{2 - \frac{\rho(x,y)^2}{t_0^2}} + \frac12 - \frac{\cos(\theta)}{2} \right)^{-\frac{d-1}{2}} \right] \right] \nonumber \\
& \left.+ \left(\frac{t}{t_0} - \sqrt{2 - \frac{\rho(x,y)^2}{t_0^2}} + 1 \right) \left(1 - \frac{\rho(x,y)^2}{t_0^2} \right)^{-\frac{d+3}{2}} \right\}. \label{Equ Lem Prelim 7}
\end{align}

Hereby, the last summand only exists for $\sqrt{2 - \rho(x,y)^2/t_0^2} -1 \leq t/t_0$ and if $\rho(x,y)/t_0$ is so close to $1$ that $\sqrt{2 - \rho(x,y)^2/t_0^2} -1 \leq \cos(\theta)$ then $\sqrt{2 - \rho(x,y)^2/t_0^2} -1$ in the last summand has to be replaced by $\cos(\theta).$
Furthermore, everything before the last summand has to be replaced by $0$ if $\cos(\theta) \geq \sqrt{2 - \rho(x,y)^2/t_0^2} -1.$\\

\noindent\underline{Case ii)} $t_0^2 \leq \rho(x,y)^2 \leq 2 t_0^2.$

Then only $s_{+,-}$ lies in $[0,t_0].$
If $s \leq s_{+,-},$ then $|s^2 - t_0^2 + \rho(x,y)^2| \geq 2st_0$ and if $s \geq s_{+,-},$ then $|s^2 - t_0^2 + \rho(x,y)^2| \leq 2st_0.$
Then we obtain

\begin{align}
& \int_{\cos(\theta)t_0}^t \left| \frac{\partial}{\partial s}k_{e^{i\theta}t_0 + s}(x,y) \right| ds \lesssim \nonumber \\
& \frac{(\cos(\theta))^{-\beta}}{\mu(B(x,t_0))} \left[ \int_{\cos(\theta)t_0}^{t_0 (1 - \sqrt{2 - \rho(x,y)^2/t_0^2})} [s^2 - t_0^2 + \rho(x,y)^2]^{-\frac{d+1}{2}} t_0^{d+1} \frac{ds}{s} \right. \nonumber \\
& \left. + \int_{t_0 (1 - \sqrt{2 - \rho(x,y)^2/t_0^2})}^t t_0^{d+1} (st_0)^{-\frac{d+1}{2}} \frac{ds}{s} \right] \nonumber \\
& \cong \frac{(\cos(\theta))^{-\beta}}{\mu(B(x,t_0))} \left\{ \left[\log\frac{1 -\sqrt{2 - \rho(x,y)^2/t_0^2}}{\cos(\theta)}\right]_+\left(\frac{\rho(x,y)^2}{t_0^2} - 1 \right)^{-\frac{d+1}{2}} \right. \nonumber \\
&\left. + \left[ \left(1 - \sqrt{2 - \rho(x,y)^2/t_0^2} \right)^{-\frac{d+1}{2}} - (t/t_0)^{-\frac{d+1}{2}} \right]_+ \right\}, \label{Equ Lem Prelim 8}
\end{align}
where $[x]_+ = \max(x,0).$
Note that in the last expression above, in the first summand, $1 -  \sqrt{2-\rho(x,y)^2/t_0^2}$ has to be replaced by $t/t_0$ if
$t/t_0 \leq 1 - \sqrt{2-\rho(x,y)^2/t_0^2}.$
In the second summand, $1 - \sqrt{2-\rho(x,y)^2/t_0^2}$ has to be replaced by $\cos(\theta)$ if
$1 - \sqrt{2 - \rho(x,y)^2/t_0^2} \leq \cos(\theta).$ \\

\noindent\underline{Case iii)} $\rho(x,y)^2 \geq 2 t_0^2.$

Then none of $s_{\pm,\pm}$ is real.
We always have $|s^2 - t_0^2 + \rho(x,y)^2| \geq 2st_0.$
Then we obtain

\begin{align}
& \int_{\cos(\theta)t_0}^t \left| \frac{\partial}{\partial s}k_{e^{i\theta}t_0 + s}(x,y) \right| ds \lesssim \frac{(\cos(\theta))^{-\beta}}{\mu(B(x,t_0))} \int_{\cos(\theta) t_0}^t t_0^{d+1} |s^2 - t_0^2 + \rho(x,y)^2|^{-\frac{d+1}{2}} \frac{ds}{s} \nonumber \\
& \lesssim \frac{(\cos(\theta))^{-\beta}}{\mu(B(x,t_0))} \left[ \log\left(\frac{t}{t_0 \cos(\theta)}\right) \right]_+ \left(\frac{\rho(x,y)^2}{t_0^2} - 1\right)^{-\frac{d+1}{2}}. \label{Equ Lem Prelim 9}
\end{align}

Summarizing \eqref{Equ Lem Prelim 5}, \eqref{Equ Lem Prelim 6}, \eqref{Equ Lem Prelim 7}, \eqref{Equ Lem Prelim 8}, \eqref{Equ Lem Prelim 9}, we finally obtain the claimed estimate of the lemma.
\end{proof}

\begin{proof}[Proof of Lemma \ref{Lem Prelim 2}]
We start with estimating $K := |k_{e^{i\theta} 2^j t_0 + t}(x,y) - k_{e^{i\theta} 2^j t_0}(x,y)|$ for a fixed $j \in \Z$ such that $2^jt_0 \geq t,$
hereby using Lemma \ref{Lem Prelim 1}.
Write in short $R = \rho(x,y)^2 / (2^j t_0)^2.$
We distinguish the following fifteen cases Iai), Iaii), Iaiii), Iaiv), Iav), Ibi), Ibii), Ic), Idi), Idii), Idiii), Ie), IIa), IIb+c), IId+e),
which depend on the values of $\theta,2^j t_0,t$ and $\rho(x,y).$
Here, case I stands for $\cos(\theta) 2^j t_0 \leq t,$ case II stands for $\cos(\theta) 2^j t_0 \geq t,$ case a stands for $R \leq 1 - \cos(\theta),$ b for $1 - \cos(\theta) \leq R \leq 1,$ c for $1 \leq R \leq 1 + \cos(\theta),$ d for $1 + \cos(\theta) \leq R \leq 2$ and e for $2 \leq R.$\\

\noindent\underline{Case Iai)} $\cos(\theta) 2^j t_0 \leq t,\: R \leq 1 - \cos(\theta)$ and $\cos(\theta) \leq \sqrt{2 - R} - 1 \leq \frac12 \frac{t}{2^j t_0}.$

Then with Lemma \ref{Lem Prelim 1}, 
\begin{align*}
K & \lesssim \frac{(\cos(\theta))^{-\beta}}{\mu(B(x,2^j t_0))} \left\{ (1 - R)^{-\frac{d+1}{2}} \left( 1 + \log\left( \frac{\frac12 [\sqrt{2 -R} -1 + \cos(\theta)]}{\cos(\theta)} \right) \right) \right. \\
& + [\sqrt{2-R}-1 + \cos(\theta)]^{-1} \left[(\sqrt{1-R} - \sqrt{2-R}+1)^{-\frac{d-1}{2}} \right.\\
& \left.\left. - \left(\sqrt{1-R}-\frac12 \sqrt{2-R} + \frac12 -\frac{\cos(\theta)}{2}\right)^{-\frac{d-1}{2}}\right] ( 1 - R)^{-\frac{d+1}{4}} + (1-R)^{-\frac{d+1}{2}} \right\}.
\end{align*}

\noindent\underline{Case Iaii)} $\cos(\theta) 2^j t_0 \leq t,\: R \leq 1 - \cos(\theta)$ and $\max(\cos(\theta),\frac12 \frac{t}{2^j t_0}) \leq \sqrt{2 - R} - 1 \leq \frac{t}{2^j t_0}.$

\begin{align*}
K & \lesssim \frac{(\cos(\theta))^{-\beta}}{\mu(B(x,2^j t_0))} \left\{ (1-R)^{-\frac{d+1}{2}} \left( 1 + \log\left( \frac{\frac12[\sqrt{2-R} -1 + \cos(\theta)]}{\cos(\theta)}\right)\right) \right. \\
& +(1-R)^{-\frac{d+1}{4}-1} \left[ (\sqrt{1-R}-\sqrt{2-R}+1)^{-\frac{d-1}{2}} \right.\\
& \left. - \left(\sqrt{1-R} - \frac12 \sqrt{2-R} + \frac12 - \frac{\cos(\theta)}{2}\right)^{-\frac{d-1}{2}} \right] \\
& \left.+ \left(\frac{t}{2^j t_0} - \sqrt{2-R} +1\right)(1-R)^{-\frac{d+3}{2}} \right\}
\end{align*}

\noindent\underline{Case Iaiii)} $\cos(\theta) 2^j t_0 \leq t,\: R \leq 1 - \cos(\theta)$ and $\max(\cos(\theta),\frac{t}{2^jt_0}) = \frac{t}{2^j t_0} \leq \sqrt{2-R} -1.$

\begin{align*}
K & \lesssim \frac{(\cos(\theta))^{-\beta}}{\mu(B(x,2^j t_0))} \left\{ (1-R)^{-\frac{d+1}{2}} \left( 1 + \log\left( \frac{\frac12[\sqrt{2-R}-1 + \cos(\theta)]}{\cos(\theta)}\right)\right) \right. \\
& + (1-R)^{-\frac{d+1}{4}-1} \left[ (\sqrt{1-R} - \sqrt{2-R} + 1)^{-\frac{d-1}{2}} \right. \\
& \left.\left.- \left( \sqrt{1-R} - \frac12 \sqrt{2-R} + \frac12 - \frac{\cos(\theta)}{2} \right)^{-\frac{d-1}{2}} \right] \right\}.
\end{align*}

\noindent\underline{Case Iaiv)} $\cos(\theta) 2^j t_0 \leq t,\: R \leq 1 - \cos(\theta)$ and $\frac12 \frac{t}{2^j t_0} \leq \sqrt{2-R} -1  \leq \cos(\theta).$

\begin{align*}
K & \lesssim \frac{(\cos(\theta))^{-\beta}}{\mu(B(x,2^j t_0))} \left\{ (1-R)^{-\frac{d+1}{2}} + \left( \frac{t}{2^j t_0} - \sqrt{2-R} +1\right) (1-R)^{-\frac{d+3}{2}} \right\}
\end{align*}

\noindent\underline{Case Iav)} $\cos(\theta) 2^j t_0 \leq t,\: R \leq 1 - \cos(\theta)$ and $\sqrt{2-R}-1 \leq \min(\cos(\theta),\frac12 \frac{t}{2^jt_0}).$

\begin{align*}
K & \lesssim \frac{(\cos(\theta))^{-\beta}}{\mu(B(x,2^j t_0))} (1-R)^{-\frac{d+1}{2}}
\end{align*}

\noindent\underline{Case Ib)} $\cos(\theta) 2^j t_0 \leq t$ and $1 - \cos(\theta) \leq R \leq 1.$

In this case, we have $\sqrt{2-R} -1 \leq \frac12 \cos(\theta) + o(\cos(\theta)),$ so that always $\sqrt{2-R} - 1 \leq \cos(\theta)$ if $|\theta|$ is sufficiently close to $\frac{\pi}{2}.$

\noindent\underline{Case Ibi)} $\cos(\theta) 2^j t_0 \leq t,\:1 - \cos(\theta) \leq R \leq 1$ and $\frac12 \frac{t}{2^j t_0} \leq \sqrt{2-R}-1.$

\begin{align*}
K & \lesssim \frac{(\cos(\theta))^{-\beta}}{\mu(B(x,2^j t_0))} \left\{ (\cos(\theta))^{-\frac{d+1}{2}} + (\cos(\theta))^{-\frac{d+1}{2}} - \left(\frac{t}{2^j t_0} \right)^{-\frac{d+1}{2}} \right\} \\
& \lesssim  \frac{(\cos(\theta))^{-\beta}}{\mu(B(x,2^j t_0))} (\cos(\theta))^{-\frac{d+1}{2}}
\end{align*}

\noindent\underline{Case Ibii)} $\cos(\theta) 2^j t_0 \leq t,\:1 - \cos(\theta) \leq R \leq 1$ and $\sqrt{2-R} -1 \leq \frac12 \frac{t}{2^jt_0}.$

\begin{align*}
K & \lesssim \frac{(\cos(\theta))^{-\beta}}{\mu(B(x,2^j t_0))} (\cos(\theta))^{-\frac{d+1}{2}}
\end{align*}

\noindent\underline{Case Ic)} $\cos(\theta) 2^j t_0 \leq t,\: 1 \leq R \leq 1 + \cos(\theta).$

In this case, we have $0 \leq 1 - \sqrt{2-R} \leq \frac12 \cos(\theta) + o(\cos(\theta)).$ 
Then

\begin{align*}
K & \lesssim \frac{(\cos(\theta))^{-\beta}}{\mu(B(x,2^j t_0))} \left\{ (\cos(\theta))^{-\frac{d+1}{2}} + (\cos(\theta))^{-\frac{d+1}{2}} - \left(\frac{t}{2^j t_0} \right)^{-\frac{d+1}{2}} \right\} \\
& \lesssim  \frac{(\cos(\theta))^{-\beta}}{\mu(B(x,2^j t_0))} (\cos(\theta))^{-\frac{d+1}{2}}
\end{align*}

\noindent\underline{Case Id)} $\cos(\theta) 2^j t_0 \leq t$ and $1 + \cos(\theta) \leq R \leq 2.$

In this case, $-1 \leq 1 - R \leq - \cos(\theta)$, so that $1 - (R-1) \leq \sqrt{1 + (1-R)} \leq 1 - \frac12(R-1)$ and thus $\frac12 (R-1) \leq 1 - \sqrt{2-R} \leq R-1.$

\noindent\underline{Case Idi)} $\cos(\theta) 2^j t_0 \leq t,\:1 + \cos(\theta) \leq R \leq 2$ and $\cos(\theta) \leq 1 - \sqrt{2-R} \leq \frac{t}{2^j t_0}.$

\begin{align*}
K & \lesssim \frac{(\cos(\theta))^{-\beta}}{\mu(B(x,2^j t_0))} \left\{(R-1)^{-\frac{d+1}{2}} \left(1 + \log\left( \frac{1 - \sqrt{2-R}}{\cos(\theta)}\right)\right) \right. \\
& \left.+ (1-\sqrt{2-R})^{-\frac{d+1}{2}} - \left(\frac{t}{2^j t_0}\right)^{-\frac{d+1}{2}} \right\} \\
& \cong \frac{(\cos(\theta))^{-\beta}}{\mu(B(x,2^j t_0))} (R-1)^{-\frac{d+1}{2}} \left(1 + \log\left(\frac{R-1}{\cos(\theta)}\right) \right)
\end{align*}

\noindent\underline{Case Idii)} $\cos(\theta) 2^j t_0 \leq t,\: 1 + \cos(\theta) \leq R \leq 2$ and $\max(\cos(\theta),\frac{t}{2^j t_0}) = \frac{t}{2^j t_0} \leq 1 - \sqrt{2-R}.$

\begin{align*}
K & \lesssim \frac{(\cos(\theta))^{-\beta}}{\mu(B(x,2^j t_0))} (R-1)^{-\frac{d+1}{2}} \left( 1 + \log\left( \frac{t}{2^jt_0 \cos(\theta)} \right) \right)
\end{align*}

\noindent\underline{Case Idiii)} $\cos(\theta) 2^j t_0 \leq t,\: 1 + \cos(\theta) \leq R \leq 2$ and $1 - \sqrt{2-R} \leq \cos(\theta).$

\begin{align*}
K & \lesssim \frac{(\cos(\theta))^{-\beta}}{\mu(B(x,2^j t_0))} \left\{ (R-1)^{-\frac{d+1}{2}} + (\cos(\theta))^{-\frac{d+1}{2}} - \left( \frac{t}{2^j t_0} \right)^{-\frac{d+1}{2}} \right\} \\
& \cong \frac{(\cos(\theta))^{-\beta}}{\mu(B(x,2^j t_0))} (\cos(\theta))^{-\frac{d+1}{2}}
\end{align*}

\noindent\underline{Case Ie)} $\cos(\theta) 2^j t_0 \leq t$ and $R \geq 2.$

\begin{align*}
K & \lesssim \frac{(\cos(\theta))^{-\beta}}{\mu(B(x,2^j t_0))} (R-1)^{-\frac{d+1}{2}} \left( 1 + \log\left( \frac{t}{2^j t_0 \cos(\theta)} \right) \right)
\end{align*}

\noindent\underline{Case IIa)} $\cos(\theta) 2^j t_0 \geq t$ and $R \leq 1 - \cos(\theta).$

\begin{align*}
K & \lesssim \frac{(\cos(\theta))^{-\beta}}{\mu(B(x,2^j t_0))} \frac{t}{\cos(\theta) 2^j t_0} (1-R)^{-\frac{d+1}{2}}
\end{align*}

\noindent\underline{Case IIb+c)} $\cos(\theta) 2^j t_0 \geq t$ and $1 - \cos(\theta) \leq R \leq 1 + \cos(\theta).$

\begin{align*}
K & \lesssim \frac{(\cos(\theta))^{-\beta}}{\mu(B(x,2^j t_0))} \frac{t}{\cos(\theta) 2^j t_0} (\cos(\theta))^{-\frac{d+1}{2}}
\end{align*}

\noindent\underline{Case IId+e)} $\cos(\theta) 2^j t_0 \geq t$ and $R \geq 1 + \cos(\theta).$

\begin{align*}
K & \lesssim \frac{(\cos(\theta))^{-\beta}}{\mu(B(x,2^j t_0))} \frac{t}{\cos(\theta) 2^j t_0} (R-1)^{-\frac{d+1}{2}}
\end{align*}

Now we want to prove part 1. of the lemma, i.e. to estimate 
\begin{equation}\label{Equ Lem Prelim 10}
\sup_{j \in \Z:\: 2^j t_0 \geq t} |k_{e^{i\theta} 2^j t_0 + t}(x,y) - k_{e^{i\theta} 2^j t_0}(x,y)|
\end{equation}
in the case that $|\rho(x,y)^2/(2^l t_0)^2 - 1| \leq \cos(\theta)$ for some $l \leq L_{\max}.$
We claim that in \eqref{Equ Lem Prelim 10}, the supremum is essentially attained for $j = l,$ more precisely, that \eqref{Equ Lem Prelim 10} can be estimated by the above estimate for $j=l$ (which has to be in one of the four cases I or II, b or c).

For $j < l,$ we have $\frac{\rho(x,y)^2}{(2^j t_0)^2} = \frac{\rho(x,y)^2}{(2^l t_0)^2} 2^{2l - 2j} \geq (1-\cos(\theta))\cdot 4,$ so that case e applies.
We note again $R = \frac{\rho(x,y)^2}{(2^j t_0)^2},$ and moreover $K_j =  |k_{e^{i\theta} 2^j t_0 + t}(x,y) - k_{e^{i\theta} 2^j t_0}(x,y)|,$ and $M_l$ the right hand side of the estimate obtained for $K_l.$
Note that since $j < l < L_{\max},$ we have $\mu(B(x,2^j t_0)) \cong (2^j t_0)^d$ by \eqref{Equ add prop space 1}.
For $\cos(\theta) 2^j t_0 \leq t,$ we have by case Ie), 
\begin{align*}
K_j & \lesssim \frac{(\cos(\theta))^{-\beta}}{\mu(B(x,2^j t_0))} (R-1)^{-\frac{d+1}{2}} \left( 1 + \log\left( \frac{t}{2^j t_0 \cos(\theta)} \right) \right) \\
& \lesssim \frac{(\cos(\theta))^{-\beta}}{(2^j t_0)^d} \frac{(2^j t_0)^{d+1}}{\rho(x,y)^{d+1}} \left( 1 + \log\left( \frac{t}{2^j t_0 \cos(\theta)} \right) \right) \\
& \cong (\cos(\theta))^{-\beta} 2^{j-l} (2^l t_0)^{-d} \left( 1 + \log\left( \frac{t}{2^j t_0 \cos(\theta) } \right) \right).
\end{align*}
If $\cos(\theta) 2^j t_0 \leq t \leq \cos(\theta) 2^l t_0,$ then we have $K_j \lesssim M_l \Longleftarrow 2^{j-l}(1 + \log \frac{t}{2^j t_0 \cos(\theta)}) \lesssim \frac{t}{t_0} 2^{-l} (\cos(\theta))^{-\frac{d+3}{2}} \Longleftrightarrow 2^j t_0 \cos(\theta)/t ( 1+  \log \frac{t}{2^j t_0} + |\log \cos(\theta)|) \lesssim ( \cos(\theta) )^{-\frac{d+1}{2}},$ which is true, since $2^j t_0 \cos(\theta)/t \leq 1$ and $\log\frac{t}{2^j t_0} \leq 0.$
If $\cos(\theta) 2^l t_0 \leq t,$ then we have $K_j \lesssim M_l \Longleftarrow 2^{j-l} ( 1 + \log \frac{t}{2^j t_0 \cos(\theta)} ) \lesssim (\cos(\theta))^{-\frac{d+1}{2}} \Longleftrightarrow 2^{j-l} ( 1 + \log \frac{t}{2^j t_0} + |\log \cos(\theta)|) \lesssim (\cos(\theta))^{-\frac{d+1}{2}},$ which is true, since $2^{j-l} \leq 1$ and $\log\frac{t}{2^j t_0} \leq 0.$

For $\cos(\theta) 2^j t_0 \geq t,$ we have by case IIe),
\begin{align*}
K_j & \lesssim \frac{(\cos(\theta))^{-\beta}}{\mu(B(x,2^j t_0))} \frac{t}{\cos(\theta) 2^j t_0} (R-1)^{-\frac{d+1}{2}} \\
& \cong \frac{(\cos(\theta))^{-\beta}}{(2^j t_0)^d} \frac{t}{\cos(\theta) 2^j t_0} \frac{(2^j t_0)^{d+1}}{\rho(x,y)^{d+1}} \\
& \cong \frac{(\cos(\theta))^{-\beta}}{\mu(B(x,2^l t_0))} \frac{t}{\cos(\theta) 2^l t_0}.
\end{align*}
This is indeed majorized by $M_l,$ since $\cos(\theta) 2^l t_0 \geq t,$ and thus, 
\[M_l = \frac{(\cos(\theta))^{-\beta}}{\mu(B(x,2^lt_0))} \frac{t}{\cos(\theta) 2^l t_0} (\cos(\theta))^{-\frac{d+1}{2}}.\]

For $j > l,$ we have $\frac{\rho(x,y)^2}{(2^j t_0)^2} = \frac{\rho(x,y)^2}{(2^l t_0)^2} \cdot 2^{2l - 2j} \leq (1 + \cos(\theta))\cdot \frac14,$ so that case a applies.
We have $1-R \cong 1,\: \sqrt{2-R}-1 \cong 1,\: \sqrt{1-R} - (\sqrt{2-R}-1) \cong 1,\;\sqrt{1-R} - \frac12 \sqrt{2-R} + \frac12 - \frac{\cos(\theta)}{2} \cong 1.$
Thus, in case Iai),
\begin{align*}
K_j & \lesssim \frac{(\cos(\theta))^{-\beta}}{\mu(B(x,2^j t_0))}(1 + |\log\cos(\theta)|) \\
& \lesssim \frac{(\cos(\theta))^{-\beta}}{\mu(B(x,2^l t_0))} (\cos(\theta))^{-\frac{d+1}{2}} \cong M_l.
\end{align*}
In case Iaii),
\begin{align*}
K_j & \lesssim \frac{(\cos(\theta))^{-\beta}}{\mu(B(x,2^j t_0))}(1 + |\log\cos(\theta)| + \frac{t}{2^j t_0} - (\sqrt{2-R}-1)) \\
& \lesssim \frac{(\cos(\theta))^{-\beta}}{\mu(B(x,2^l t_0))} (\cos(\theta))^{-\frac{d+1}{2}} \cong M_l.
\end{align*}
In case Iaiii), $K_j \lesssim \frac{(\cos(\theta))^{-\beta}}{\mu(B(x,2^j t_0))} (1 + |\log\cos(\theta)|) \lesssim M_l.$

\noindent Cases Iaiv) and Iav) cannot appear here, since $\sqrt{2-R}-1 > \cos(\theta).$

\noindent In case IIa), 
\begin{align*}
K_j & \lesssim \frac{(\cos(\theta))^{-\beta}}{\mu(B(x,2^j t_0))} \frac{t}{\cos(\theta) 2^j t_0} \cdot 1 \\
& \lesssim \frac{(\cos(\theta))^{-\beta}}{\mu(B(x,2^l t_0))} \min(1,\frac{t}{\cos(\theta) 2^l t_0}) (\cos(\theta))^{-\frac{d+1}{2}} \cong M_l.
\end{align*}

Looking up $M_l$ in the four cases I or II, b or c, now yields part 1. of the lemma.\\

Now for the proof of part 2.
Suppose first that $1 + \cos(\theta) \leq \frac{\rho(x,y)^2}{(2^l t_0)^2} \leq 2.$
Then $M_l$ is given by one of the four cases Idi),Idii),Idiii) and IId).
We will show that $\sup_{j \in \Z:\:2^j t_0 \geq t} K_j \lesssim M_l + \frac{(\cos(\theta))^{-\beta}}{\mu(B(x,2^lt_0))} ( 1 + |\log\cos(\theta)| ),$
the logarithmic term only appearing in case Idii).

Suppose first that $M_l$ is given by case Idi).
Consider a $j < l.$
Then $\frac{\rho(x,y)^2}{(2^j t_0)^2} = \frac{\rho(x,y)^2}{(2^l t_0)^2} \cdot 2^{2l-2j} \geq 4,$ so that for $K_j,$ case Ie) or IIe) applies.
In fact, one is never in the case IIe), since then $\cos(\theta) 2^l t_0 \leq t,\: \cos(\theta) 2^j t_0 \geq t$ and $j < l.$
In case Ie), we have
\begin{align*}
K_j & \lesssim \frac{(\cos(\theta))^{-\beta}}{\mu(B(x,2^j t_0))} R^{-\frac{d+1}{2}} \left(1 + \log\left( \frac{t}{2^j t_0 \cos(\theta)} \right) \right) \\
& \cong (\cos(\theta))^{-\beta} 2^j t_0 \rho(x,y)^{-(d+1)} \left(1 + \log\left( \frac{t}{2^j t_0 \cos(\theta)} \right) \right) \\
& \cong (\cos(\theta))^{-\beta} 2^{j-l} (2^l t_0)^{-d} \left(1 + \log\left( \frac{t}{2^j t_0 \cos(\theta)} \right) \right).
\end{align*}
On the other hand, 
\begin{align*}
M_l & = \frac{(\cos(\theta))^{-\beta}}{\mu(B(x,2^l t_0))} \left\{ \left(\rho(x,y)^2 / (2^l t_0)^2 - 1\right)^{-\frac{d+1}{2}} \left( 1 + \log\left( \frac{1 - \sqrt{2 - \rho(x,y)^2/(2^l t_0)^2}}{\cos(\theta)} \right)\right) \right. \\
& \left. + \left(1 - \sqrt{2 - \rho(x,y)^2/(2^l t_0)^2}\right)^{-\frac{d+1}{2}} - \left(\frac{t}{2^l t_0}\right)^{-\frac{d+1}{2}} \right\}.
\end{align*}
Now  we have $K_j \lesssim M_l$ if
\[1 + | \log(\cos(\theta)) | \lesssim \left(\rho(x,y)^2/(2^l t_0)^2 -1\right)^{-\frac{d+1}{2}} \left( 1 + \log \frac{\rho(x,y)^2/(2^l t_0)^2 - 1}{\cos(\theta)}\right).\]
An elementary calculation shows that the minimum of the right hand side for $1 + \cos(\theta) \leq \frac{\rho(x,y)^2}{(2^l t_0)^2} \leq 2$ is equivalent to $1 + |\log(\cos(\theta))|.$
Thus $K_j \lesssim M_l$ for $j < l.$
Now consider $j > l.$
Then $\frac{\rho(x,y)^2}{(2^j t_0)^2} = \frac{\rho(x,y)^2}{(2^l t_0)^2} 2^{2l - 2j} \leq \frac12,$ so that case a) applies for $K_j.$
Now we have $1 - R \cong 1,\: \sqrt{2-R}-1\cong1 ,\: \sqrt{1-R} - (\sqrt{2-R}-1) \cong 1,$ and $\sqrt{1-R} - \frac12(\sqrt{2-R}-1) - \frac{\cos(\theta)}{2} \cong 1.$
With this, we obtain easily in case Iai) that
$K_j \lesssim \frac{(\cos(\theta))^{-\beta}}{\mu(B(x,2^j t_0))} ( 1 + |\log(\cos(\theta))| ) \lesssim \frac{(\cos(\theta))^{-\beta}}{\mu(B(x,2^l t_0))} ( 1 + |\log(\cos(\theta))| ) \lesssim M_l.$
In case Iaii), we also have $K_j \lesssim \frac{(\cos(\theta))^{-\beta}}{\mu(B(x,2^j t_0))} ( 1 + |\log(\cos(\theta))| ) \lesssim M_l,$ and the cases Iaiii), Iaiv) and Iav) can be handled in the same way.
In case IIa), we have $K_j \lesssim \frac{(\cos(\theta))^{-\beta}}{\mu(B(x,2^j t_0))} \frac{t}{\cos(\theta) 2^j t_0} \lesssim M_l.$
Thus $\sup_{j: 2^j t_0 \geq t} K_j \lesssim M_l$ if $M_l$ is given by case Idi).\\

Now suppose that $M_l$ is given by case Idii).
Take first a $j < l.$
Then $\rho(x,y)^2/(2^j t_0)^2 \geq 4,$ so that for $K_j,$ case e) applies.
In case Ie), we have
\begin{align*}
K_j & \lesssim \frac{(\cos(\theta))^{-\beta}}{\mu(B(x,2^j t_0))} \left(\frac{\rho(x,y)^2}{(2^j t_0)^2} \right)^{-\frac{d+1}{2}} \left( 1 + \log\frac{t}{2^j t_0 \cos(\theta)} \right) \\
& \lesssim (\cos(\theta))^{-\beta} 2^{j-l} (2^l t_0)^{-d} \left( 1 + \log \frac{t}{2^j t_0 \cos(\theta)} \right) \\
& \lesssim (\cos(\theta))^{-\beta} (2^l t_0)^{-d} \left( 1 + \log\frac{t}{2^l t_0 \cos(\theta)} \right) \cong M_l.
\end{align*}
Again, case IIe) cannot appear since then, $\cos(\theta) 2^l t_0 \leq t,\: \cos(\theta) 2^j t_0 \geq t$ and $j < l.$

Now consider a $j > l.$
We have $\rho(x,y)^2/(2^j t_0)^2 \leq \frac12,$ so that case a) applies for the estimate of $K_j.$
In case Iai), we have $K_j \lesssim \frac{(\cos(\theta))^{-\beta}}{\mu(B(x,2^j t_0))} (1+|\log(\cos(\theta))|) \lesssim \frac{(\cos(\theta))^{-\beta}}{\mu(B(x,2^l t_0))} (1+|\log(\cos(\theta))|).$
The same estimate holds in the cases Iaii) and Iaiii).
In case Iaiv), we have $K_j \lesssim \frac{(\cos(\theta))^{-\beta}}{\mu(B(x,2^j t_0))} \lesssim M_l,$ and similarly, also in the cases, Iav) and IIa), $K_j \lesssim M_l$ holds.
We thus have $\sup_{j: 2^j t_0 \geq t} K_j \lesssim M_l + \frac{(\cos(\theta))^{-\beta}}{\mu(B(x,2^l t_0))}(1+|\log(\cos(\theta))|),$ if $M_l$ is given by case Idii).\\

Suppose now that $M_l$ is given by case Idiii).
Then $M_l \cong \frac{(\cos(\theta))^{-\beta}}{\mu(B(x,2^l t_0))} (\cos(\theta))^{-\frac{d+1}{2}}.$
Consider a $j<l.$
Again, for the estimate of $K_j,$ case e) applies, and case IIe) is ruled out by the constraints on $j$ and $l.$
We have as above
\begin{align*}
K_j & \lesssim (\cos(\theta))^{-\beta} 2^{j-l} (2^l t_0)^{-d} \left( 1 + \log \frac{t}{2^j t_0 \cos(\theta)} \right) \\
& \lesssim \frac{(\cos(\theta))^{-\beta}}{\mu(B(x,2^l t_0))} ( 1 + \log \frac{t}{2^j t_0} + |\log(\cos(\theta))| ) \lesssim M_l.
\end{align*}
Consider now a $j>l.$
Again $\rho(x,y)^2/(2^j t_0)^2 \leq \frac12,$ so that for $K_j,$ case a) applies, and again, several expressions involving $R$ appearing in this case are equivalent to $1.$
This gives in case Iai), Iaii) and Iaiii), $K_j \lesssim \frac{(\cos(\theta))^{-\beta}}{\mu(B(x,2^l t_0))}  (1 + | \log(\cos(\theta))|) \lesssim M_l,$
and in case Iaiv) and Iav), $K_j \lesssim \frac{(\cos(\theta))^{-\beta}}{\mu(B(x,2^l t_0))} \lesssim M_l.$
In case IIa), we have $K_j \lesssim \frac{(\cos(\theta))^{-\beta}}{\mu(B(x,2^l t_0))} \frac{t}{\cos(\theta) 2^j t_0} \lesssim M_l.$
Thus $\sup_{j: 2^j t_0 \geq t} K_j \lesssim M_l$ if $M_l$ is given by case Idiii).\\

Suppose now that $M_l$ is given by case IId).
Then 
\[M_l =  \frac{(\cos(\theta))^{-\beta}}{\mu(B(x,2^l t_0))} \frac{t}{\cos(\theta) 2^l t_0} \left( \frac{\rho(x,y)^2}{(2^l t_0)^2} - 1 \right)^{-\frac{d+1}{2}}.\]
Consider $j < l,$ so that for $K_j,$ the estimate in case e) applies.
In case Ie), we have
\begin{align*}
K_j & \lesssim \frac{(\cos(\theta))^{-\beta}}{\mu(B(x,2^l t_0))} 2^{(l-j)d} \left(\frac{\rho(x,y)^2}{(2^jt_0)^2}\right)^{-\frac{d+1}{2}} \left( 1 + \log\frac{t}{2^j t_0 \cos(\theta)} \right) \\
& \lesssim \frac{(\cos(\theta))^{-\beta}}{\mu(B(x,2^l t_0))} 2^{-l+j} \left( 1 + \log \frac{t}{2^j t_0 \cos(\theta)} \right) \\
& \lesssim \frac{(\cos(\theta))^{-\beta}}{\mu(B(x,2^l t_0))} \frac{t}{\cos(\theta) 2^l t_0} \lesssim M_l.
\end{align*}
In case IIe), we have
\[ K_j \lesssim (\cos(\theta))^{-\beta} (2^j t_0)^{-d} \frac{t}{\cos(\theta) 2^j t_0} \left( \frac{2^l}{2^j} \right)^{-(d+1)} \!\!\!\!\!\!\!\!\!\!\! \cong (2^l t_0)^{-d} 2^{j-l} \frac{t}{\cos(\theta) 2^j t_0} \lesssim M_l. \]
Consider now $j > l,$ so that for $K_j,$ the estimate in case a) applies.
The cases Iai) - Iav) cannot appear due to restrictions on $j,l$ and $\cos(\theta)t_0/t$ similar as before.
In case IIa), we have $K_j \lesssim \frac{(\cos(\theta))^{-\beta}}{\mu(B(x,2^j t_0))} \frac{t}{\cos(\theta) 2^j t_0} \lesssim  \frac{(\cos(\theta))^{-\beta}}{\mu(B(x,2^l t_0))} \lesssim M_l.$
Thus $\sup_{j: 2^j t_0 \geq t} K_j \lesssim M_l$ if $M_l$ is given by case IId).
We have proved part 2. of the lemma.\\

Now for the proof of part 3. of the lemma.
We proceed similarly as before.
The expression $M_l$ is given by one of the six cases Iai) - Iav) or IIa).
Suppose first that case Iai) applies.
Consider $j < l.$
Then $R \geq 2,$ so that for $K_j,$ case e) applies.
In case Ie), we have 
\begin{align*}
K_j & \lesssim (\cos(\theta))^{-\beta} (2^j t_0)^{-d} \frac{(2^l t_0)^{-(d+1)}}{(2^j t_0)^{-(d+1)}} \left( 1 + \log \frac{t}{2^j t_0 \cos(\theta)} \right) \\
& \lesssim (\cos(\theta))^{-\beta} (2^l t_0)^{-d} ( 1 + |\log(\cos(\theta))| ).
\end{align*}
Case IIe) cannot appear.
Now consider $j > l.$
Then for $K_j,$ case a) applies.
As several times before, diverse terms in $R$ are equivalent to 1.
In cases Iai), Iaii) and Iaiii), we have $K_j \lesssim \frac{(\cos(\theta))^{-\beta}}{\mu(B(x,2^j t_0))} (1 + |\log(\cos(\theta))| ) \leq \frac{(\cos(\theta))^{-\beta}}{\mu(B(x,2^l t_0))} (1 + |\log(\cos(\theta))| ).$
In cases Iaiv) and Iav), we have $K_j \lesssim \frac{(\cos(\theta))^{-\beta}}{\mu(B(x,2^j t_0))} \lesssim M_l,$ and in case IIa),
we have $K_j \lesssim \frac{(\cos(\theta))^{-\beta}}{\mu(B(x,2^j t_0))} \frac{t}{\cos(\theta) 2^j t_0} \lesssim \frac{(\cos(\theta))^{-\beta}}{\mu(B(x,2^l t_0))}.$
Thus, if $M_l$ is given by Iai), we have $\sup_{j: 2^j t_0 \geq t} K_j \lesssim M_l + \frac{(\cos(\theta))^{-\beta}}{\mu(B(x,2^l t_0))} ( 1 + |\log(\cos(\theta))| ).$\\

Suppose that $M_l$ is given by case Iaii).
Consider a $j < l.$
Then $K_j$ is estimated by case Ie), and we have $K_j \lesssim (\cos(\theta))^{-\beta} (2^l t_0)^{-d} ( 1 + |\log(\cos(\theta))| ).$
Consider a $j > l.$
Then $K_j$ is estimated by case a), and we have $K_j \lesssim \frac{(\cos(\theta))^{-\beta}}{\mu(B(x,2^l t_0))} ( 1 + | \log(\cos(\theta)) | ).$
Thus, if $M_l$ is given by Iaii), we have again $\sup_{j: 2^j t_0 \geq t} K_j \lesssim M_l + \frac{(\cos(\theta))^{-\beta}}{\mu(B(x,2^l t_0))} ( 1 + |\log(\cos(\theta))| ).$\\

Similarly, we obtain in the cases that $M_l$ is given by Iaiii), Iaiv) or Iav) that $\sup_{j: 2^j t_0 \geq t} K_j \lesssim M_l + \frac{(\cos(\theta))^{-\beta}}{\mu(B(x,2^l t_0))} ( 1 + |\log(\cos(\theta))| ).$\\

If $M_l$ is given by IIa), we have $M_l = \frac{(\cos(\theta))^{-\beta}}{\mu(B(x,2^l t_0))} \frac{t}{\cos(\theta) 2^l t_0} \left( 1 - \frac{\rho(x,y)^2}{(2^l t_0)^2} \right)^{-\frac{d+1}{2}}.$
Consider a $j < l.$
Then $R \geq 2,$ so that case Ie) or IIe) applies for $K_j.$
In case Ie), we have
\begin{align*}
K_j & \lesssim (\cos(\theta))^{-\beta} (2^j t_0)^{-d} \frac{(2^j t_0)^{d+1}}{(2^l t_0)^{d+1}} \left( 1 + \log \frac{t}{2^j t_0 \cos(\theta)} \right) \\
& \lesssim (\cos(\theta))^{-\beta} (2^l t_0)^{-d} 2^{j-l} \left( 1 + \log \frac{t}{2^j t_0 \cos(\theta)} \right) \lesssim M_l.
\end{align*}
In case IIe), we have $K_j \lesssim (2^l t_0)^{-d} 2^{j-l} \frac{t}{\cos(\theta) 2^j t_0} \lesssim M_l.$
Consider now a $j > l.$
Then for $K_j,$ case a) applies.
Cases Iai) - Iav) cannot occur due to restrictions on $j,l$ and $\cos(\theta) t_0/t$ as several times before.
In case IIa), we have $K_j \lesssim \frac{(\cos(\theta))^{-\beta}}{\mu(B(x,2^j t_0))} \frac{t}{\cos(\theta) 2^j t_0} \lesssim  \frac{(\cos(\theta))^{-\beta}}{\mu(B(x,2^l t_0))} \frac{t}{\cos(\theta) 2^l t_0} \lesssim M_l.$
Thus if $M_l$ is given by IIa), we have $\sup_{j: 2^j t_0 \geq t} K_j \lesssim M_l.$
\end{proof}

\section*{Acknowledgement}
The author would like to thank the harmonic analysis group of the Australian National University in Canberra for their kind hospitality, in particular Thierry Coulhon, Dorothee Frey, Pierre Portal and Adam Sikora.
During the stay at Canberra the main theorem could be improved.

\end{document}